\documentclass[journal]{IEEEtran}  
\usepackage{graphicx, float}
\usepackage{amsmath} 
\usepackage{amsthm}
\usepackage{amssymb}
\usepackage{booktabs}
\usepackage{siunitx}
\usepackage{cite}
\usepackage[linesnumbered, ruled]{algorithm2e}
\makeatletter
\renewcommand{\@algocf@capt@plain}{above}
\makeatother
\usepackage[subtle]{savetrees}
\usepackage{changes}
\definechangesauthor[name=Nawaf, color=blue]{1}
\newtheorem{theorem}{Theorem}[section]

\newtheorem*{remark}{Remark}

\ifCLASSOPTIONcompsoc
    \usepackage[caption=false, font=normalsize, labelfont=sf, textfont=sf]{subfig}
\else
\usepackage[caption=false, font=footnotesize]{subfig}
\fi

\IEEEoverridecommandlockouts                              

\title{\Large \textbf{
Voltage positioning using co-optimization of controllable grid assets in radial networks}
}

\author{Nawaf Nazir, \textit{Student Member, IEEE} and Mads Almassalkhi \textit{Senior Member, IEEE} 
\thanks{N. Nazir and M. Almassalkhi are with the Department of Electrical and Biomedical Engineering, University of Vermont, Burlington, Vermont, USA {\tt\small \{mnazir,malmassa\}@uvm.edu} Support from U.S. Department of Energy award number DE-EE0008006 is gratefully acknowledged.}%
}

\begin{document}

\maketitle
\thispagestyle{empty}
\pagestyle{empty}

\begin{abstract}
With increasing penetration of solar PV, some distribution feeders are experiencing highly variable net-load flows and even reverse flows.
To optimize distribution systems under such conditions, the scheduling of mechanical devices, such as OLTCs and capacitor banks, needs to take into account forecasted solar PV and actual grid conditions. 
However, these legacy switching assets are operated on a daily or hourly timescale, due to the wear and tear associated with mechanical switching, which makes them unsuitable for real-time control.
Therefore, there is a natural timescale-separation between these slower mechanical assets and the responsive nature of inverter-based resources. In this paper, we present a network admissible convex formulation for holistically scheduling controllable grid assets to position voltage optimally against solar PV. An optimal hourly schedule is presented that utilizes mechanical resources 
to position the predicted voltages close to nominal values, while minimizing the use of inverter-based resources (i.e., DERs), making them available for control at a faster time-scale (after the uncertainty reveals itself). A convex, inner approximation of the OPF problem is adapted to a mixed-integer linear program that minimizes voltage deviations from nominal (i.e., maximizes voltage margins). The resulting OPF solution respects all the network constraints and is, hence, robust against modeling simplifications.
Simulation based analysis on 
IEEE distribution feeders validates the approach.
\end{abstract}

\section{Introduction}\label{sec:introduction}
With the increasing penetration of renewable resources in the distribution grid, maintaining system voltages within acceptable limits (i.e., minimizing voltage deviations), is a major challenge~\cite{driesen2006distributed,bravo2015distributed}. The intermittent nature of solar energy can cause under and over-voltages in the system~\cite{stetz2013improved,NERC2017,bayer2018german,guangya2015voltage} leading to unacceptable operation. However, solar PV resources are inverter interfaced and can provide responsive reactive power resources, which can be used in active network management~\cite{basso2004ieee}. Besides these inverter-interfaced resources, the distribution grid also includes traditional mechanical devices, such as on-load tap changing (OLTC) transformers, cap banks, reactors, etc. These discrete mechanical assets are subject to physical wear and tear and, thus, are usually only operated a few times during the day with heuristic open-loop policies~\cite{sanghi2003chemistry}. However, with increasing solar PV penetration, it becomes important to optimize the schedule for the mechanical assets against bidirectional and variable power flows~\cite{kersulis2016renewable}. However, the mechanical switching is not suitable for real-time conditions and control and should, therefore, be utilized on slower timescales to position the predicted voltage profile (i.e., increase voltage margins) against predicted solar PV generation. In fact, inverter-interfaced assets, such as solar PV generation and battery storage, can effectively supply controllable reactive resources appropriate for these faster time-scales. Therefore, there is a natural timescale-separation between (slow) mechanical and (fast) inverter-based controllable grid assets. DER resources on slow time-scale act as a form of reactive reserve, allowing the DERs to fully participate in valuable market services on a fast timescale. This way mechanical assets maximize margins and optimize value of DERs. This leads to the challenge of co-optimization of different types of controllable reactive power resources. 
Thus, for the scheduling on slower time-scales, it is  desirable to maximize utilization of the mechanical assets to keep voltages close to desirable nominal values while using as little as possible of the responsive inverter-interfaced reactive resources. This effectively prioritizes the responsive reactive resources for the faster time-scales to counter variability in net-load (demand minus solar PV). 

The aim of this paper is then to present a convex OPF formulation where the objective function seeks to minimize the deviation of the predicted nodal voltages from their nominal values. The nonlinear power flow  equations  relate  the  voltages in  the network with the complex power injections. 
Traditionally, for distribution system ACOPF formulations, the nonlinear \textit{DistFlow} model is used,  which considers a branch flow model~\cite{baran1989optimal}. Recently,  convex relaxation techniques have been developed to formulate and solve the OPF problem to global optimality~\cite{lavaei2012zero,lavaei2014geometry}.  These convex relaxations provide a lower bound on the globally optimal AC solution. Several works in literature such as~\cite{gan2015exact} have shown that under some conditions these relaxations can be exact and the solution of the relaxed convex problem is the global optimum of the original AC OPF problem. However, these conditions fail to hold under reverse power flows from extreme solar PV, which engenders a non-zero duality gap and a non-physical solution~\cite{huang2017sufficient}.
Further simplications from convex models lead to linearized AC power flows, which have also been shown to be accurate in certain applications. Of particular interest to distribution system OPF is an extension of the \textit{LinDistFlow} model to an unbalanced linearized load flow model, \textit{Dist3Flow}, that is obtained by linearization and certain assumptions on the per-phase imbalances~\cite{sankur2016linearized,robbins2016optimal}. However, solving linearized OPF problems, even though computationally efficient, do not provide guarantees on feasibility or bounds on optimality with respect to the original nonlinear formulation. This paper uses a convex approximation of the power flow equations that results in a network-admissible solution, i.e., all physical network limits are respected at (global) optimality, while solving in polynominal time. Hence, the method is robust against modeling errors introduced from approximations of the non-linear power flow equations.

Discrete devices like the capacitor/reactor banks and line regulators (ON/OFF)  and  load-tap-changing (LTC) transformers are  an  integral  part  of  distribution  system operations. Due to the discrete nature of these devices, including them into an  optimization  problem renders  the  problem  NP-hard~\cite{Zhu2016optimal}. To incorporate discrete devices into convex OPF formulations, the McCormick relaxations~\cite{mccormick1976computability} and linearization techniques have  been  used  to incorporate  these  devices~\cite{nazir2018receding,briglia2017distribution}. In~\cite{Zhu2016optimal}, the authors use SDPs (semi-definite programs) to capture the transformer ratios and then the solutions are rounded to the nearest discrete tap values, whereas in~\cite{yang2017optimal} the load tap changers and shunt capacitors are both modeled by linear constraints using discrete variables, facilitating the linearly constrained mixed-integer formulation. However, this rounding can cause infeasibility issues, which are analyzed in~\cite{shukla2019} and the authors provide an MISOCP formulation, which is computationally tractable and converges to a feasible optimal solution. This paper builds upon these works, but leverages the notion that discrete devices and continuous resources  can  offer  their  flexibility  at  different  time-scales, which gives rise to a natural prioritization of reactive power resources.

This paper focuses on optimizing discrete control assets in the grid to maximize both the voltage margins and the availability of reactive reserves for the faster timescales. This maximization of voltage margins\footnote{Minimizing voltage deviations from nominal can be viewed as maximizing voltage margins.} is illustrated in Fig.~\ref{fig:volt_gradient} which depicts larger voltage margins as we move closer to the nominal. Recent works in literature such as~\cite{bolognani2015distributed,zhou2017incentive,baker2018network}    have developed control schemes that achieve voltage regulation through dispatch of flexible resources in real-time. The work in this paper could provide prediction schedules for voltage control at faster timescales.

\begin{figure}[h]
\centering
\includegraphics[width=0.4\textwidth]{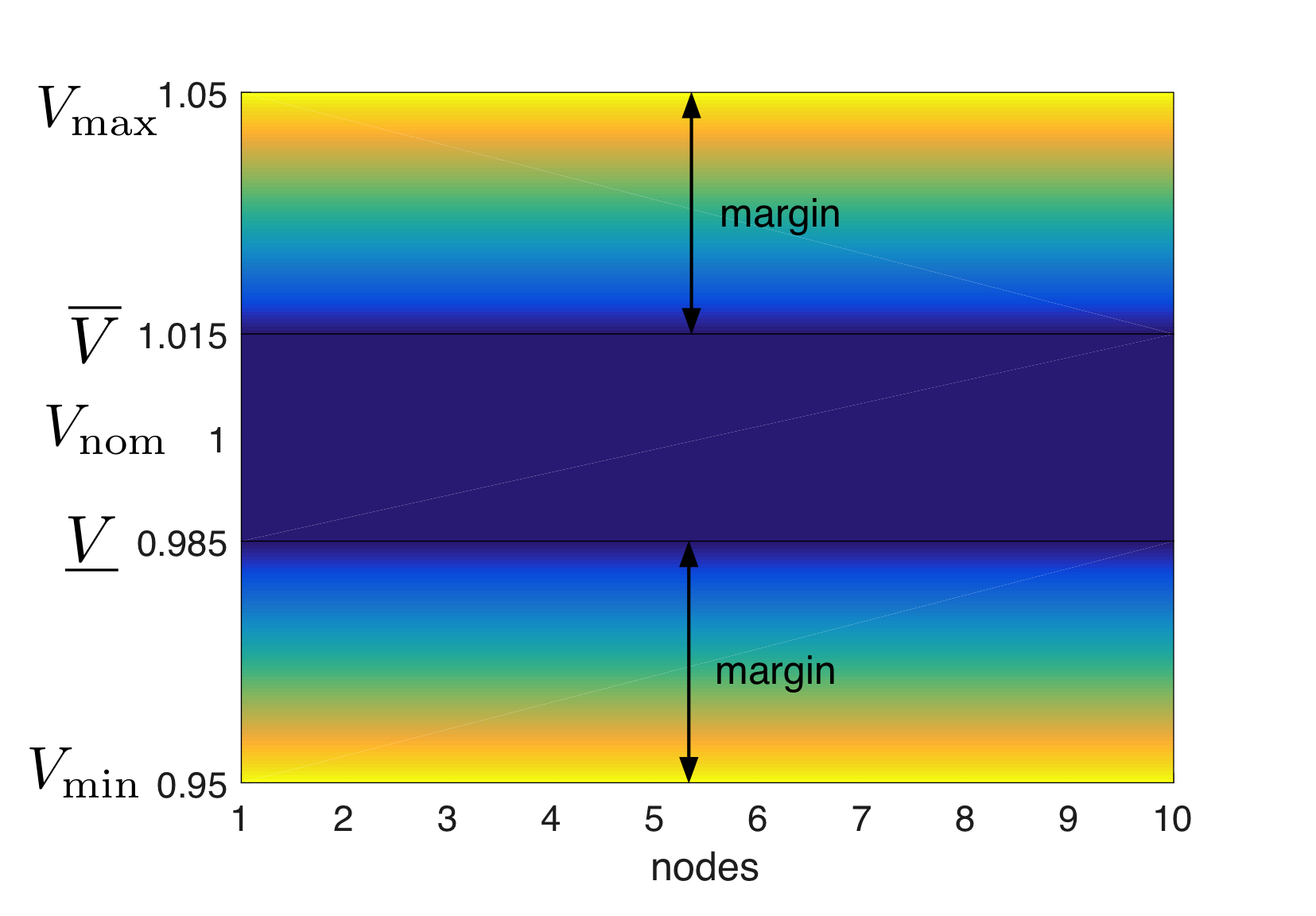}
\caption{\label{fig:volt_gradient}Illustrating the relationship between voltage margin, nominal voltage and voltage bounds. Lighter colors represent larger margins.}
\end{figure}


In general, employing convex relaxations with an objective that minimizes voltage deviations will lead to a non-zero duality gap~\cite{li2017non}, due to being non-monotonic. More general conditions for the exactness of the convex relaxation are shown in~\cite{li2018convex,gan2015exact}. This makes it challenging to use convex relaxations to formulate the optimization problem. In many applications, providing network admissibility guarantees are more valuable than solving to a globally optimal solution~\cite{Dan_convex_rest}. This paper uses the convex inner approximation method of the OPF problem that exhibits computational solve times similar to that of linear formulations with the added and crucial benefit that the formulation guarantees admissible solutions.  Furthermore, the utilization of reactive power from flexible inverter-interfaced DERs should be minimized, so that this resource can be better utilized at the faster time-scale. Previous work on minimizing both voltage deviations and reactive power use has been  shown in~\cite{wu2017distributed}, where a trade-off parameter is used between the two competing objectives. Unlike~\cite{wu2017distributed}, this paper considers control scheme for the integration of existing discrete mechanical assets and flexible inverters and provides a systematic method to select the trade-off parameter. 
To summarize, main contributions of this paper are the following:
\begin{itemize}
\item A convex inner OPF formulation is developed for the problem of minimizing voltage deviations from nominal  in a distribution system with guarantees on admissiblility and scalability.

\item A voltage positioning optimization (VPO) method is developed that holistically optimizes the schedule of discrete mechanical assets while systematically minimizing the need for continuous inverter-interfaced reactive DERs. 
\item Simulation based analysis is leveraged to select trade-off parameters between the use of continuous reactive resources and voltage margins, which are then utilized to validate the performance on IEEE test feeders.
\end{itemize}

The rest of the paper is organized as follows:
Section~\ref{sec:VPP} develops the voltage positioning optimization (VPO) formulation to include discrete mechanical assets as a mixed integer program in order to position the nodal voltages. Section~\ref{sec:Opt_form} develops the mathematical formulation of the convex inner approximation OPF problem that is then used in the VPO problem to obtain a MILP based VPO.  Simulation results on IEEE test feeders are discussed in Section~\ref{sec:sim_results} and finally conclusions and future directions are summarized in Section~\ref{sec:conclusion}.

\section{Voltage Positioning Optimization}\label{sec:VPP}

This section develops the voltage positioning optimization problem as a mixed integer program (MIP). The nonlinearity associated with modeling discrete mechanical devices, such as On-load tap changers (OLTCs) and capacitor banks (CBs), is expressed with an equivalent piece-wise linear formulation to engender the MIP.

\subsection{Distribution Grid Model}\label{sec:math_model}
 Let $\mathbb{R}$ be the set of real numbers, $\mathbb{Z}$ be the set of integers and $\mathbb{N}$ be the set of natural numbers. Consider a radial balanced distribution network, shown in Fig.~\ref{fig:radial_network}, as a graph $\mathcal{G}=\{\mathcal{N},\mathcal{E}\}$, where $\mathcal{N}$ is the set of nodes and $\mathcal{E}$ is the set of branches, such that $(i,j)\in \mathcal{E}$, if nodes $i,j\in \mathcal{N}$ are connected, and $|\mathcal{E}|=n, \ |\mathcal{N}|=n+1$. Node $0$ which is assumed to be the substation node with a fixed voltage $V_0$ and define $\mathcal{N^+}:=\mathcal{N}\setminus \{0\}$. Let $B\in \mathbb{R}^{(n+1)\times n}$ be the \textit{incidence matrix} of the undirected graph $\mathcal{G}$ relating the branches in $\mathcal{E}$ to the nodes in $\mathcal{N}$, such that the entry at $(i,j)$ of $B$ is $1$ if the $i$-th node is connected to the $j$-th branch and otherwise $0$. If $V_i \in \mathbb{C}$ and $V_j\in \mathbb{C}$ are the voltage phasors at nodes $i$ and $j$ and $I_{ij} \in \mathbb{C}$ is the current phasor in branch $(i,j)\in \mathcal{E}$, then $v_i:=|V_i|^2$, $v_j:=|V_j|^2$ and $l_{ij}:=|I_{ij}|^2$. Let $P_{ij}$ be the sending end active power flow from node $i$ to $j$, $Q_{ij}$ be the sending end reactive power flow from node $i$ to $j$, $p_i$ be the active power injection and $q_i$ be the reactive power injection, into node $i \in \mathcal{N^+}$, $r_{ij}$ and $x_{ij}$ be the resistance and reactance of the branch $(i,j)\in \mathcal{E}$ and $z_{ij}=r_{ij}+\mathbf{j}x_{ij}$ be the impedance. Then, for the radial distribution network, the relation between node voltages and power flows is given by the  \textit{DistFlow} equations:
\vspace{2mm}
\begin{align}
v_i=&v_j+2r_{ij}P_{ij}+2x_{ij}Q_{ij}-|z_{ij}|^2l_{ij} \quad \forall (i,j) \in \mathcal{E} \label{eq:volt_rel}\\
P_{ij}=&p_i+\sum_{h:h\rightarrow i}(P_{hi}-r_{hi}l_{hi}) \quad \forall (i,j) \in \mathcal{E} \label{eq:real_power_rel}\\
Q_{ij}=&q_i+\sum_{h:h\rightarrow i}(Q_{hi}-x_{hi}l_{hi}) \quad \forall (i,j) \in \mathcal{E} \label{eq:reac_power_rel}\\
l_{ij}v_i=&P_{ij}^2+Q_{ij}^2 \quad \forall (i,j)\in \mathcal{E} \label{eq:curr_rel1}
\end{align}

\begin{figure}[h]
\centering
\includegraphics[width=0.42\textwidth]{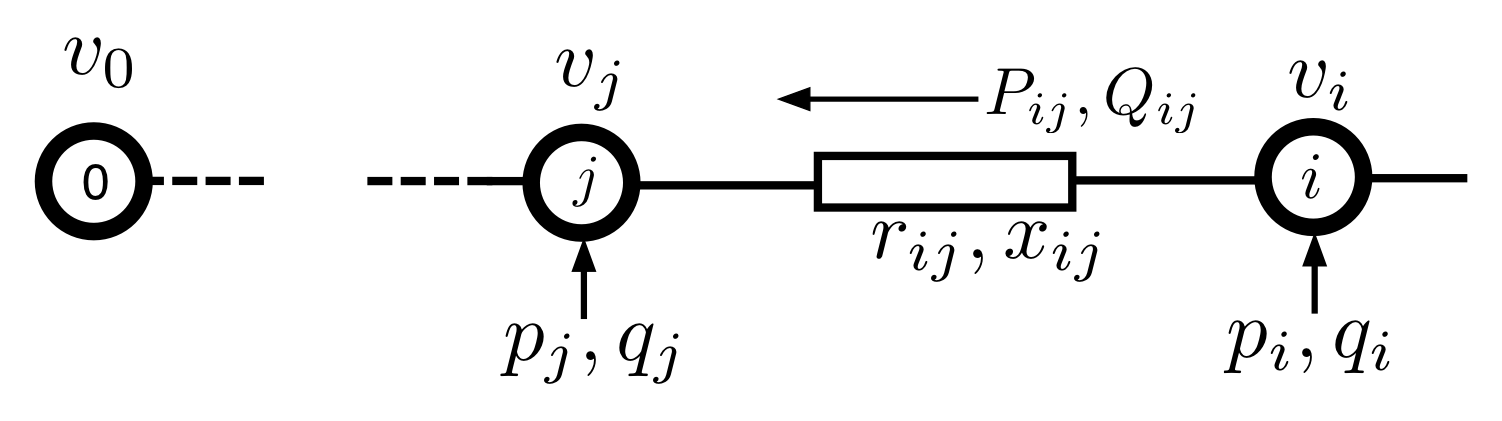}
\caption{\label{fig:radial_network} Diagram of a radial distribution network from~\cite{heidari2017non}.}
\vspace{-15pt}
\end{figure}

Clearly, the line losses in~\eqref{eq:curr_rel1} are nonlinear, and since it is an equality constraint, this makes the \textit{DistFlow} model non convex. In the remainder of this section, we develop a mathematical model of the radial network that expresses the constrained variables as a linear function of the power injections and the branch currents. Through this approach, we are able to separate the model into linear and nonlinear components. This will form the basis for our approach to bound the nonlinear terms, which will then provide the convex inner approximation of the power flow equations.

From the incidence matrix $\mathbf{B}$ of the radial network and following the method adopted in~\cite{heidari2017non}, \eqref{eq:real_power_rel} and \eqref{eq:reac_power_rel} can be expressed through the following matrix equations:
\begin{align}
    \mathbf{P}=\mathbf{p}+\mathbf{AP}-\mathbf{ARl} \qquad \mathbf{Q}=\mathbf{q}+\mathbf{AQ}-\mathbf{AXl}, \label{eq:P_matrix} 
\end{align}
where $\mathbf{P}=[P_{ij}]_{(i,j)\in \mathcal{E}}$, $\mathbf{Q}=[Q_{ij}]_{(i,j)\in \mathcal{E}}$, $\mathbf{p}=[p_i]_{i\in \mathcal{N^+}}$, $\mathbf{q}=[q_i]_{i \in \mathcal{N^+}}$, $\mathbf{R}=\text{diag}\{r_{ij}\}_{(i,j)\in \mathcal{E}}$, $\mathbf{X}=\text{diag}\{x_{ij}\}_{(i,j)\in \mathcal{E}}$, $\mathbf{l}=[l_{ij}]_{(i,j)\in \mathcal{E}}$ and $\mathbf{A}=[\mathbf{0}_n \quad \mathbf{I}_n]\mathbf{B}-\mathbf{I}_n$, where $\mathbf{I}_n$ is the $n\times n$ identity matrix and $\mathbf{0}_n$ is a column vector of $n$ rows.

Defining  $\mathbf{C}=(\mathbf{I}_n-\mathbf{A})^{-1}$, $\mathbf{D_{\text{R}}}=(\mathbf{I}_n-\mathbf{A})^{-1}\mathbf{AR}$, and $\mathbf{D_{\text{X}}}=(\mathbf{I}_n-\mathbf{A})^{-1}\mathbf{AX}$, allows us to simplify \eqref{eq:P_matrix} to:
\begin{align} \label{eq:P_relation}
    \mathbf{P}=\mathbf{Cp}-\mathbf{D_{\text{R}}l} \qquad \mathbf{Q}=\mathbf{Cq}-\mathbf{D_{\text{X}}l},
\end{align}

\begin{remark}
The matrix $(\mathbf{I}_n-\mathbf{A})$ is nonsingular since $\mathbf{I}_n-\mathbf{A}=2\mathbf{I}_n-[\mathbf{0}_n \quad \mathbf{I}_n]\mathbf{B}=2\mathbf{I}_n-\mathbf{B}_n$, where $\mathbf{B}_n:=[\mathbf{0}_n \quad \mathbf{I}_n]\mathbf{B}$ is the $n\times n$ matrix obtained by removing the first row of $\mathbf{B}$. For a radial network, the vertices and edges can always be ordered in such a way that $\mathbf{B}$ and $\mathbf{B}_n$ are upper triangular with $\text{diag}(\mathbf{B}_n) = \mathbf{1}_n$, which implies that $2\mathbf{I}_n-\mathbf{B}_n$ is also upper triangular and $\text{diag}(2\mathbf{I}_n-\mathbf{B}_n)=\mathbf{1}_n$. Thus,   $\det(2\mathbf{I}_n-\mathbf{B}^{'})=1>0$ and $\mathbf{I}_n-\mathbf{A}$ is  non-singular.
\end{remark}

Similarly, \eqref{eq:volt_rel} can be applied recursively to the distribution network in Fig.~\ref{fig:radial_network} to get the matrix equation:
\begin{align}\label{eq:volt_matrix_rel}
    [v_i-v_j]_{(i,j)\in \mathcal{E}}=2(\mathbf{RP}+\mathbf{XQ})-\mathbf{Z}^2\mathbf{l}
\end{align}
where $\mathbf{Z}^2:=\text{diag}\{z_{ij}^2\}_{(i,j)\in \mathcal{E}}$. Based on the incidence matrix $\mathbf{B}$, the left hand side of \eqref{eq:volt_matrix_rel} can be formulated in terms of the fixed head node voltage as:
\begin{align}\label{eq:volt_transform}
    \mathbf{C}^\top [v_i-v_j]_{(i,j)\in \mathcal{E}}=\mathbf{V}-v_{\text{0}} \mathbf{1}_n
\end{align}
where $\mathbf{V}:=[v_i]_{i\in \mathcal{N^+}}$. Based on \eqref{eq:volt_transform}, \eqref{eq:volt_matrix_rel} can be expressed as:
\begin{align}\label{eq:volt_matrix_2}
    \mathbf{V}=v_{\text{0}} \mathbf{1}_n+2(\mathbf{C}^\top \mathbf{RP}+\mathbf{C}^\top \mathbf{XQ})-\mathbf{C}^\top \mathbf{Z}^2\mathbf{l}
\end{align}
Substituting \eqref{eq:P_relation} into \eqref{eq:volt_matrix_2}, we obtain a compact relation between voltage and power injections shown below.
\begin{align}\label{eq:final_volt_rel1}
    \mathbf{V}=v_{\text{0}}\mathbf{1}_n+\mathbf{M_{\text{p}}p}+\mathbf{M_{\text{q}}q}-\mathbf{Hl}
\end{align}
where $\mathbf{M_{\text{p}}}=2\mathbf{C}^\top \mathbf{RC}$, \quad $\mathbf{M_{\text{q}}}=2\mathbf{C}^\top \mathbf{XC}$ and \newline $\mathbf{H}=\mathbf{C}^\top (2(\mathbf{RD_{\text{R}}}+\mathbf{XD_{\text{X}}})+\mathbf{Z}^2)$
\begin{remark}
The matrix $\mathbf{H}$ is non-negative, when the underlying distribution network is either inductive ($\mathbf{X}$ is non-negative), capacitive ($\mathbf{X}$ is non-positive) or purely resistive ($\mathbf{X}$ is zero matrix). This fact helps in obtaining the convex inner approximation described later in the paper in Section~III. Substituting the values of $\mathbf{C}$, $\mathbf{D_{\text{R}}}$ and $\mathbf{D_{\text{X}}}$ into the expression of $\mathbf{H}$, gives:
    $$\mathbf{H}=(\mathbf{I}_n-\mathbf{A})^{-\top}[2(\mathbf{R}(\mathbf{I}_n-\mathbf{A})^{-1}\mathbf{AR}+\mathbf{X}(\mathbf{I}_n-\mathbf{A})^{-1}\mathbf{AX})+\mathbf{Z}^2].$$ To show $\mathbf{H}$ is non-negative, we just need to focus on $\mathbf{A}$ and $(\mathbf{I}_n-\mathbf{A})^{-1}$. Due to the definition, $\mathbf{A}$ is non-negative and $\mathbf{I}_n-\mathbf{A}$ has positive diagonal entries and non-positive off-diagonal entries and is, hence, a Z-matrix. Also, $\mathbf{I}_n-\mathbf{A} = 2\mathbf{I}_n - \mathbf{B}_n$ and $2\mathbf{I}_n - \mathbf{B}_n$ is an upper triangular matrix.  Hence its eigenvalues are positive, so it is also a non-singular M-matrix (i.e., a Z-matrix whose eigenvalues have non-negative real part). Non-singular M-matrices are a subset of a class of inverse-positive matrices, i.e., matrices with inverses belonging to the class of non-negative matrices (all the elements are either equal to or greater than zeros)~\cite[Corollary 3.2]{fujimoto2004two}.
  Hence, $(\mathbf{I}_n-\mathbf{A})^{-1}$ is a non-negative matrix. As $\mathbf{A}$ is also a non-negative matrix, then $\mathbf{H}$ is clearly non-negative whenever matrix $\mathbf{R}$ is non-negative and either $\mathbf{X}$ is non-negative (i.e., all lines are inductive), $\mathbf{X}$ is non-positive (i.e., all lines are capacitive) or $\mathbf{X}$ is zero (i.e., all lines are purely resistive).
\end{remark}

Apart from the nonlinear relation \eqref{eq:curr_rel1} of $\mathbf{l}$ to $\mathbf{P}$, $\mathbf{Q}$ and $\mathbf{V}$, \eqref{eq:P_relation} and \eqref{eq:final_volt_rel1} is a linear relationship between the nodal power injections $\mathbf{p}$, $\mathbf{q}$, the branch power flows $\mathbf{P}$ ,$\mathbf{Q}$ and node voltages $\mathbf{V}$. The nonlinearity in the network is represented by~\eqref{eq:curr_rel1}, as the current term $\mathbf{l}$ is related to the power injections and node voltages in a nonlinear fashion. Including this term into the optimization model would render the optimization problem non-convex, however, neglecting this term could result in an inadmissible linear OPF solution. In the next section, we model the discrete grid resources such as OLTCs and capacitor banks, which will then be used to formulate the voltage positioning optimiazation problem.

\subsection{Discrete device nomenclature}\label{sec:notation}
Consider the distribution grid defined in Section~\ref{sec:math_model}, where 
$\mathcal{D, C}\subseteq \mathcal{N^+}$ represent the sets of nodes with DERs and capacitor banks, respectively, and $\mathcal{T}\subseteq \mathcal{E}$ is the set of branches with on-load tap change transformers (OLTC) or voltage regulators.
The tap-ratio for the OLTC/regulator at branch $m\in \mathcal{T}$ is denoted by  $t_{m}$ with the tap position defined by $n_m^{\text{tr}} \in \mathbb{Z}$, e.g. $n_m^{\text{tr}}\in \{-16,\hdots, 0,\hdots, +16\}$. The number of capacitor bank units is $n_i^{\text{cp}}\in \mathbb{Z}$ at node $i \in \mathcal{C}$.
Denote $q_{\text{g},i}$ as the controllable DER reactive power injection at node $i\in \mathcal{D}$ and $b_{i}$ as the capacitor bank admittance at node $i\in \mathcal{C}$.

\subsection{Voltage positioning optimization formulation}\label{math_formulation}

	The focus of this work is to maximize both the voltage margins and the availability of reactive reserves for the faster timescales i.e., position voltages within tighter bounds $\mathbf{\underline{V}}$ and $\mathbf{\overline{V}}$, and prioritize the use of mechanical (discrete) assets over more flexible reactive resources, $\mathbf{q_{\text{g}}}$. Hence the objective minimizes a function of $\mathbf{q_{\text{g}}}$ and the voltage deviation terms for the upper and lower bounds, $\mathbf{V_{\text{v}}^+}$ and $\mathbf{V_{\text{v}}^-}$, respectively. The VPO formulation is described next.

Given a radial, balanced, and single-phase equivalent representation of a distribution feeder, denote the VPO problem as (P1), which is expressed as a mixed-integer nonlinear program (MINLP) as follows:
\begin{subequations}\label{eq:P1}
\begin{align}
&\textbf{(P1)}& &\min_{q_{\text{g},i},n_m^{\text{tr}},n_i^{\text{cp}}} \sum_{i=1}^nq_{\text{g},i}^2+\alpha_i(V^+_{\text{v},i}+V^-_{\text{v},i})& \label{eq:P1_obj}
\end{align}
\text{subject to: }
\begin{align}
     \mathbf{V}=&v_{\text{0}}\mathbf{1}_n+\mathbf{M_{\text{p}}p}+\mathbf{M_{\text{q}}q}-\mathbf{Hl}\label{eq:final_volt_rel}\\
 l_{ij}v_i=&P_{ij}^2+Q_{ij}^2  \qquad \forall (i,j)\in \mathcal{E} \label{eq:curr_rel}\\
\mathbf{q}=&\mathbf{q_{\text{g}}}-\mathbf{Q_{\text{L}}}+\mathbf{Q^{\text{cp}}}\label{eq:P1_q_rel}\\
 0=& v_i - t_{m}^2 v_j \quad\forall m \in \mathcal{T} \label{eq:P1_xmer_relation} \\
  0=& t_m - (1+\tau_mn_m^{\text{tr}}) \quad \forall m\in \mathcal{T}\label{eq:trans_tap_rel}\\
  0= & Q_i^{\text{cp}}-v_ib_i \quad\forall i\in \mathcal{C}\label{eq:P1_cap_volt_rel}\\
  0=& b_i - y_{\text{c},i}n_i^{\text{cp}} \quad \forall i\in \mathcal{C}\label{eq:P1_cap_num_rel}\\
    \mathbf{V_{\text{min}}} \leq & \mathbf{V} \leq \mathbf{V_{\text{max}}}\label{eq:P1_vbound}\\
    \mathbf{\underline{V}}-\mathbf{V^-_{\text{v}}}\leq & \mathbf{V} \leq \mathbf{\overline{V}}+\mathbf{V^+_{\text{v}}}\label{eq:P1_sl}\\
    \mathbf{V^+_{\text{v}}}\ge & 0,\,\, \mathbf{V^-_{\text{v}}}\ge 0\label{eq:P1_sl_pos}\\
    \underline{q_{\text{g},i}}\leq & q_{\text{g},i}\leq \overline{q_{\text{g},i}} \quad\forall i\in \mathcal{D} \label{eq:P1_qbound} \\
    \underline{n_m^{\text{tr}}}\leq & n_m^{\text{tr}}\leq \overline{n_m^{\text{tr}}} \quad \forall m \in \mathcal{T}\label{eq:P1_xmer_bound} \\
    \underline{n_i^{\text{cp}}}\leq & n_i^{\text{cp}}\leq \overline{n_i^{\text{cp}}} \quad \forall i \in \mathcal{C}\label{eq:P1_cap_bound}\\
     n_m^{\text{tr}},\,\,& n_i^{\text{cp}} \in \mathbb{Z} \quad \forall m\in \mathcal{T}, \forall i\in \mathcal{C} \label{eq:P1_int_constraint}
\end{align}
\end{subequations}
where $q_{\text{g},i}$ is the DER reactive power generation at node $i$ and $V^+_{\text{v},i}$ and $V^-_{\text{v},i}$ represents the voltage violation terms for the upper and lower bound respectively, at node $i$. The parameter $\alpha$ is chosen to trade-off between the use of flexible reactive resources and maximizing voltage margins. The equality constraints \eqref{eq:final_volt_rel} and \eqref{eq:curr_rel} represent the power flow equations relating the voltages and currents in the network to the power injections, whereas  \eqref{eq:P1_q_rel} represents the nodal reactive power balance, with $\mathbf{Q_{\text{L}}}$ being the reactive net-demand and $\mathbf{p}=-\mathbf{P_{\text{L}}}$ with $\mathbf{P_{\text{L}}}$ being the active net-demand. The constraints \eqref{eq:P1_xmer_relation} and \eqref{eq:trans_tap_rel} define the relation between the tap ratio and the tap position with $\tau_m \in \mathbb{R}$ being the tap step, whereas the limits in~\eqref{eq:P1_xmer_bound} define bounds on OLTC tap position with $\underline{n_m^{\text{tr}}}$ and $\overline{n_m^{\text{tr}}}$ being the lower and upper tap position limit. The relation between capacitor bank admittance ($b_i$) and reactive power injected by capacitor banks ($Q_i^{\text{cp}}$) is given by~\eqref{eq:P1_cap_volt_rel}, whereas the relation between capacitor bank admittance and number of capacitor bank units with $y_{\text{c},i} \in \mathbb{R}$ being the admittance of a single capacitor bank unit is given by~\eqref{eq:P1_cap_num_rel} and \eqref{eq:P1_cap_bound} gives bounds on the capacitor bank units with $\underline{n_i^{\text{cp}}}$ and $\overline{n_i^{\text{cp}}}$ being the lower and upper bound on number of capacitor bank units. The box constraints in~\eqref{eq:P1_vbound} are the network voltage limits with $\mathbf{V_{\text{min}}}$ and $\mathbf{V_{\text{max}}}$ being the lower and upper network voltage limit. The constraint in~\eqref{eq:P1_sl} represents the tighter voltage bound constraints that seek to position the voltage close to nominal using the tighter inner voltage bounds $\mathbf{\underline{V}}$ and $\mathbf{\overline{V}}$. This ensures that the reactive power resources are utilized to position the voltage within the tighter voltage bounds. The box constraints~\eqref{eq:P1_qbound} represents the DER reactive power generation limits for each generator node with $\underline{q_{\text{g},i}}$ and $\overline{q_{\text{g},i}}$ being the lower and upper limit on generation and finally, \eqref{eq:P1_int_constraint} constrains the transformer tap positions and the number of capacitor bank units to be discrete set of integers.

(P1) represents the VPO problem for a radial distribution network. Note that nonlinear equality constraints~\eqref{eq:curr_rel}, \eqref{eq:P1_xmer_relation} and \eqref{eq:P1_cap_volt_rel} and the integer constraint
 \eqref{eq:P1_int_constraint} represent non-convex constraints and make the OPF problem NP-hard. The nonlinearity related to the transformer taps and the bilinear term for the capacitor banks are approximated with piecewise linear (PWL) constraints as shown in the next section, whereas the nonlinearity due to the powerflow equations represented by~\eqref{eq:curr_rel} is dealt with through convex inner approximation illustrated in section~\ref{sec:Opt_form}

\subsection{OLTC and capacitor bank modeling}\label{sec:Xmer_modeling}
The voltage relation between the nodes across an OLTC is given by~\eqref{eq:P1_xmer_relation} and~\eqref{eq:trans_tap_rel}.
Note that the equality constraint~\eqref{eq:P1_xmer_relation} represents a non-convex constraint and makes the OPF problem NP-hard. The nonlinearity related to the OLTC taps is approximated with piecewise linear (PWL) constraints in~\eqref{eq:xmer_pw2}-\eqref{eq:xmer_pw5} to obtain an accurate representation as described in~\cite{yang2017optimal} and summarized next.
The coupling between $v_i, v_j, t_{m}$ can be expressed as:
\begin{equation}\label{eq:xmer_pw1}
    v_i=t_{m}^2v_j \approx t_{m,0}^2v_j+\sum_{p=1}^{n_m^{\text{tr}}-\underline{n_m^{\text{tr}}}+1}\Delta t_{m,p}v_j,
\end{equation}
where $ \Delta t_{m,p}=t_{m,p}^2-t_{m,p-1}^2$, $\{t_{m,0},t_{m,1},t_{m,2},\hdots,t_{m,K}\}$ represent the fixed tap ratio settings of the OLTC connected at branch $m$ and $n_m^{\text{tr}}-\underline{n_m^{\text{tr}}}+1$ is the index of tap position $n_m^{\text{tr}}$. Next, we use binary variables $\{s^m_1,s^m_2,\hdots,s^m_K\}$ with adjacency conditions $s^m_p\geq s^m_{p+1}$, $p=1,2,\hdots,K-1$ to represent the operating status of the OLTC branch and the following group of mixed-integer linear constraints exactly describe the OLTC connected at branch $m$ in~\eqref{eq:xmer_pw1}:
\begin{subequations}
\begin{align}
v_i&=t_{m,0}^2v_j+\sum_{p=1}^K\Delta v_p^{m} \label{eq:xmer_pw2} \\
    0&\leq \Delta v_p^{m}\leq
    s^m_p\overline{v}\Delta t_{m,p}
 \label{eq:xmer_pw3}\\
\Delta t_{m,p}(v_j-(1-s^m_p)\bar{v})&\leq    \Delta v_p^{m}\leq \Delta t_{m,p}v_j \label{eq:xmer_pw4}\\
   s^m_{p+1} &\leq s^m_p,\, p=1,2,\hdots,K-1.
    \label{eq:xmer_pw5}
\end{align}
\end{subequations}
Similarly, for capacitor banks, the relation between capacitor bank admittance and number of capacitor bank units is given by~\eqref{eq:P1_cap_num_rel}.

If $Q_i^{\text{cp}}$ represents the reactive power injection from capacitor banks at node $i$, then:
\begin{align}\label{eq:cap_bank_sum}
    Q_i^{\text{cp}}=v_ib_i=\sum_{p=1}^{n_i^{\text{cp}}}(v_ib_{i,p})
\end{align}
represents the bilinearity, where $\{b_{i,1},b_{i,2},\hdots,b_{i,K}\}$ are the admissible admittance values of controllable capacitor banks at node~$i$. Similar to the formulation in \eqref{eq:xmer_pw2}-\eqref{eq:xmer_pw5}, for the capacitor bank at node $i$, \eqref{eq:cap_bank_sum} can be equivalently expressed by the following set of linear constraints~\cite{yang2017optimal}:
\begin{subequations}
\begin{align}
    Q_i^{\text{cp}}=&\sum_{p=1}^KQ_{i,p}^{\text{s}} \label{eq:cap_pw1}\\
    0\leq & Q_{i,p}^{\text{s}}\leq u^i_p\overline{v}b_{i,p} \label{eq:cap_pw2}\\
 b_{i,p}(v_i-(1-u^i_p)\bar{v})\leq &  Q_{i,p}^{\text{s}}\leq v_ib_{i,p}
 \label{eq:cap_pw3}\\
  u^i_{p+1} \leq &  u^i_p,\, p=1,2,\hdots,K-1.
\label{eq:cap_pw5}
\end{align}
\end{subequations}
where binary $\{u^i_1,u^i_2,\hdots,u^i_K\}$ represent the operating status of the capacitor bank units on node~$i$. Based on the linear modeling of OLTCs and capacitor banks in this section, we now present the mixed-integer program to solve the voltage positioning problem with the piecewise linear formulation of OLTCs and capacitor banks as shown in (P2).

\begin{subequations}\label{eq:P2}
\begin{align}
&\textbf{(P2)}& &\min_{q_{\text{g},i},s_p,u_p} \sum_{i=1}^nq_{\text{g},i}^2+\alpha_i(V^+_{\text{v},i}+V^-_{\text{v},i})& \label{eq:P2_obj}
\end{align}
\begin{align}
\text{subject to: }&
\eqref{eq:final_volt_rel}-\eqref{eq:P1_q_rel},\eqref{eq:P1_vbound}-\eqref{eq:P1_int_constraint}\\
&\eqref{eq:xmer_pw2}-\eqref{eq:xmer_pw5}, \eqref{eq:cap_pw1}-\eqref{eq:cap_pw5}
\end{align}
\end{subequations}

The VPO problem presented in (P2) is convex in the continuous variables except for the nonlinear constraint in~\eqref{eq:curr_rel}. One possible solution is to employ convex relaxation techniques to the nonlinear constraints and obtain an SDP or SOCP formulation. However, several works in literature such as~\cite{li2017non}, have shown that employing convex relaxations with an objective that minimizes voltage deviations will lead to a non-zero duality gap. On the other hand, linearized OPF techniques, even though computationally efficient, do not provide guarantees on feasibility or bounds on optimality.

To overcome these challenges associated with the nonlinearity of the power flow equations, the next section describes the convex inner approximation method of the OPF problem.

\section{Formulation of the Convex Inner Approximation} \label{sec:Opt_form}
To obtain the convex inner approximation, the approach presented bounds the non-linear terms in the power flow equations and develops an admissible model that is robust against modeling errors due to the nonlinearity. This means that the technique ensures that nodal voltages, branch power flows, and current magnitudes are within their limits at optimality.

The optimization problem (P2) is non-convex due to the constraint \eqref{eq:curr_rel}. In order to obtain an inner convex approximation of (P2), we bound the nonlinearity introduced due to \eqref{eq:curr_rel}. Let $\mathbf{l_{\text{min}}}\in \mathbb{R}^n$ and $\mathbf{l_{\text{max}}}\in \mathbb{R}^n$ be the lower and upper bound on $\mathbf{l}\in \mathbb{R}^n$, respectively. Then based on these values and provided that the matrices $\mathbf{D_R}$,  $\mathbf{D_X}$, $\mathbf{M_p}$ ,$\mathbf{M_q}$ and $\mathbf{H}$ are positive for an inductive radial network~\cite{nazir2019convex}, define:
\begin{align}
    \mathbf{V^+}:=&v_{\text{0}}\mathbf{1}_n+\mathbf{M_{\text{p}}p}+\mathbf{M_{\text{q}}q}-\mathbf{Hl_{\text{min}}}\label{eq:V_relation_1}\\
    \mathbf{V^-}:=&v_{\text{0}}\mathbf{1}_n+\mathbf{M_{\text{p}}p}+\mathbf{M_{\text{q}}q}-\mathbf{Hl_{\text{max}}}.\label{eq:V_relation_2}
\end{align}

If $\mathbf{l_{\text{min}}}$ and $\mathbf{l_{\text{max}}}$ are known, then the optimization problem (P2) can be modified to a convex inner approximation of the OPF problem. In the proceeding analysis we will provide a method to obtain an accurate representation of these bounds.
 In~\cite{nazir2019convex}, we provided conservative bounds on the nonlinearity based on worst case net-demand forecasts. In this section, we present rigorous analysis to obtain tighter lower and upper bounds on the nonlinearity using local bounds. 
 

Consider the nonlinear term in the power flow equations given by~\eqref{eq:curr_rel}. From the second-order Taylor series expansion, $l_{ij}$ can be expressed as:
\begin{align}\label{eq:T_exp}
    l_{ij} & \approx l_{ij}^0 + \mathbf{J_{ij}^\top} \mathbf{\delta_{ij}} +\frac{1}{2}\mathbf{\delta_{ij}^\top} \mathbf{H_{\text{e},ij}} \mathbf{\delta_{ij}}
\end{align}
where $l_{ij}^0$ is the value of $l_{ij}$ at the forecast net-demand and $\mathbf{\delta_{ij}}$, the Jacobian $\mathbf{J_{ij}}$ and the Hessian $\mathbf{H_{\text{e},ij}}$ are defined below.
\begin{align}\label{eq:Jacobian}
    \mathbf{\delta_{ij}}:=\begin{bmatrix} P_{ij}-P_{ij}^0\\
Q_{ij}-Q_{ij}^0\\ v_i-v_i^0\end{bmatrix} \qquad
\mathbf{J_{ij}}:=\begin{bmatrix}\frac{2P^0_{ij}}{v^0_i}\\\frac{2Q^0_{ij}}{v^0_i}\\ -\frac{(P^0_{ij})^2+(Q^0_{ij})^2}{(v^0_i)^2}\end{bmatrix}
\end{align}
\begin{align}\label{eq:Hessian}
    \mathbf{H_{\text{e},ij}}:=\begin{bmatrix} \frac{2}{v^0_i} && 0 && \frac{-2P^0_{ij}}{(v^0_i)^2}\\
    0 && \frac{2}{v^0_i} && \frac{-2Q^0_{ij}}{(v^0_i)^2}\\
    \frac{-2P^0_{ij}}{(v^0_i)^2} && \frac{-2Q^0_{ij}}{(v^0_i)^2} && 2\frac{(P^0_{ij})^2+(Q^0_{ij})^2}{(v^0_i)^3}
    \end{bmatrix}
\end{align}
 where the superscript 0, $(.)^0$, denotes the nominal values at the forecasted net demand for all variables $(.)$ in~\eqref{eq:Jacobian} and~\eqref{eq:Hessian}. The eigenvalues of the Hessian $\mathbf{H_{\text{e},ij}}$ are all non-negative, with two of the eigenvalues being strictly positive and one is zero. As the Hessian is positive semi-definite, the nonlinear function $l_{ij}$ is convex. If a function is convex then the linear approximation underbounds the nonlinear function~\cite{boyd2004convex}, i.e.,
\begin{align}
      l_{ij} & \ge l_{ij}^0 + \mathbf{J_{ij}^\top}\mathbf{\delta_{ij}} =: l_{\text{min},ij} \qquad \forall (i,j) \in \mathcal{L}\label{eq:l_lower}
\end{align}

The upper bound on the nonlinearity is obtained next and the convex inner approximation based on these bounds is presented.
Applying Taylor's theorem to the expansion, the upper bound on the nonlinear function $l_{ij}$ is given by:

\begin{align}
    |l_{ij}| & \approx |l_{ij}^0 + \mathbf{J_{ij}^\top}\mathbf{\delta_{ij}}+\frac{1}{2}\mathbf{\delta_{ij}^\top} \mathbf{H_{\text{e},ij}} \mathbf{\delta_{ij}}| \\
    & \le |l_{ij}^0| + |\mathbf{J_{ij}^\top}\mathbf{\delta_{ij}}|+|\frac{1}{2}\mathbf{\delta_{ij}^\top} \mathbf{H_{\text{e},ij}} \mathbf{\delta_{ij}}| \\
    & \le l_{ij}^0 + \max\{2|\mathbf{J_{ij}^\top}\mathbf{\delta_{ij}}|,|\mathbf{\delta_{ij}^\top} \mathbf{H_{\text{e},ij}} \mathbf{\delta_{ij}}|\} =: l_{\text{max},ij}\label{eq:l_upper}
\end{align}

Based on this upper and lower bound determined, we can now formulate the complete convex inner approximation VPO problem by modifying (P2) as:

\begin{subequations}\label{eq:P3}
\begin{align}
&\textbf{(P3)}& &\min_{q_{\text{g},i},s_p,u_p} \sum_{i=1}^nq_{\text{g},i}^2+\alpha_i(V^+_{\text{v},i}+V^-_{\text{v},i})& \label{eq:P3_obj}
\end{align}
\begin{align}
\text{subject to: }&
\eqref{eq:V_relation_1},\eqref{eq:V_relation_2},\eqref{eq:l_lower},\eqref{eq:l_upper}\\
&\eqref{eq:xmer_pw2}-\eqref{eq:xmer_pw5}, \eqref{eq:cap_pw1}-\eqref{eq:cap_pw5}\\
&\mathbf{V_{\text{min}}} \leq \mathbf{V^-} ; \mathbf{V^+}  \leq \mathbf{V_{\text{max}}}\label{eq:P3_f}\\
&\mathbf{\underline{V}}-\mathbf{V^-_{\text{v}}} \leq \mathbf{V^-} ; \mathbf{V^+}  \leq \mathbf{\overline{V}}+\mathbf{V^+_{\text{v}}}\label{eq:P4_f}\\
& \eqref{eq:P1_q_rel},\eqref{eq:P1_sl_pos}-\eqref{eq:P1_int_constraint}
\end{align}
\end{subequations}

The optimization problem (P3) represents the convex inner approximation of the VPO problem that provides a network admissible solution. This formulation includes discrete mechanical assets resulting in mixed-integer linear program (MILP) Voltage Positioning Optimization problem.

\subsection{Iterative algorithm for improving solution}
In this section, we present an iterative algorithm that achieves tighter bounds on the non-linearity. The lower and upper bounds obtained in previous section can be conservative depending upon the initial net-demand forecast. Without Algorithm~1, if we only solved (P3) once, it could result in a conservative inner approximation, which would lead to reduced performance. This is because the operating point $\mathbf{x}_0$ could be close to the no-load condition, i.e.,  $P_{ij}^0=Q_{ij}^0\approx0$, which means that the Jacobian would be close to zero per (20) and the first-order estimate of $l_{\min}$ and $l_{\max}$ would be close to $l_0$ per (22) and (25). This results in conservative feasible set for (P3) and Algorithm~1 overcomes this by successively enhancing the feasible solutions by updating the operating point and the Jacobian (and the Hessian) with the optimized decision variables, sometimes called the convex-concave procedure~\cite{boyd2004convex}. Algorithm~\ref{alg1} shows the steps involved in the proposed iterative scheme, where $\mathbf{q_{\text{g}}}^{\ast}(k)$ and $\mathbf{Q^{\text{cp}}\ast}(k)$ represent the solution of (P3) for the $k$th iteration of Algorithm~\ref{alg1}, whereas $\mathbf{q}(k)$ represents the net-reactive power injection at iteration $k$. It is assumed that the initial operating point satisfies $\mathbf{q_{\text{g}}}$ and $\mathbf{Q^{\text{cp}}}$ both being zero and, hence, $\mathbf{q}(0)=-\mathbf{Q_{\text{L}}}$. Finally the result of Algorithm~1 is given by the cumulative sum of the iterates, i.e., $\mathbf{q_{\text{g}}}=\sum_{i=1}^{k-1}\mathbf{q_{\text{g}}}^{\ast}(i)$, $\mathbf{Q^{\text{cp}}}=\sum_{i=1}^{k-1}\mathbf{Q^{\text{cp}}\ast}(i)$.

\begin{algorithm}\label{alg1}
 \caption{Successive feasible solution enhancement}
        \SetAlgoLined
        \KwResult{ $\mathbf{q_{\text{g}}}$, $\mathbf{Q^{\text{cp}}}$, $\mathbf{n^{\text{tr}}}$}
        \textbf{Input:} $\mathbf{Q_{\text{L}}}$, $f(\mathbf{x}_0)$, $\epsilon$\\
        Run AC load flow with $\mathbf{q}(0)=-\mathbf{Q_{\text{L}}}$ $\Rightarrow$ $\mathbf{J}(0),  \mathbf{H_{\text{e}}}(0)$\\
        Initialize $k=1$, $error(0)=\infty$\\
        \While{$error(k-1$) $>\epsilon$}
        {
        Solve (P3)  $\Rightarrow \mathbf{q_{\text{g}}}^{\ast}(k),  \mathbf{Q^{\text{cp}}\ast}(k), \mathbf{n^{\text{tr}}\ast}(k)$,  $f(\mathbf{x}_k^{\ast})$ \\
        Update $\sum_{i=1}^k(\mathbf{q_{\text{g}}}^{\ast}(i)+\mathbf{Q^{\text{cp}}\ast}(i))-\mathbf{Q_{\text{L}}}\Rightarrow \mathbf{q}(k)$\\
        Run AC load flow with $\mathbf{q}(k)$ $\Rightarrow$ $\mathbf{J}(k),  \mathbf{H_{\text{e}}}(k)$\\
        Update $error(k)=||f(\mathbf{x}_k^{\ast})-f(\mathbf{x}_{k-1}^{\ast})||_{\infty}$\\
        $k:=k+1$
        }
        $\mathbf{q_{\text{g}}}=\sum_{i=1}^{k-1}\mathbf{q_{\text{g}}}^{\ast}(i)$, $\mathbf{Q^{\text{cp}}}=\sum_{i=1}^{k-1}\mathbf{Q^{\text{cp}}\ast}(i)$, $\mathbf{n^{\text{tr}}}=\mathbf{n^{\text{tr}}\ast}(k-1)$
\end{algorithm}
Theorem~\ref{theorem1} proves the feasibility and convergence of solutions obtained through Algorithm~\ref{alg1}.
\begin{theorem}\label{theorem1}
Every iterate of Algorithm~1 is AC feasible and the iterates converge to a locally optimal solution.
\end{theorem}
\begin{proof}
Let $\chi$ be the feasible set of the underlying, nonconvex ACOPF from (P1) with convex objective function $f(.)$ given in (11a) and let $\mathbf{x}_0=\begin{bmatrix}\mathbf{P(p,q)} & \mathbf{Q(p,q)}
& \mathbf{V(p,q)} & \mathbf{l(p,q)}
\end{bmatrix}^{\top}$ be a feasible AC operating point, i.e., $\mathbf{x}_0\in \chi$, that depends on $\mathbf{(p,q)}$ injections. Also, define the feasible set of the convex inner approximation (P3) based on $\mathbf{x}_0$ as $\Psi_0(\mathbf{x}_0)$. Now, let $\mathbf{x}_1^{\ast}$ be the optimal solution of (P3), then $\mathbf{x}_1^{\ast}\in \Psi_0(\mathbf{x}_0)$ and $\mathbf{x}_0 \in \Psi_0(\mathbf{x}_0)$ and by definition of inner approximation $\Psi_0(\mathbf{x}_0) \subseteq \chi$. Also since $\mathbf{x}_1^{\ast}$ is the optimal solution, then $f(\mathbf{x}_1^{\ast})\le f(\mathbf{x}_0)$. This process can be repeated so that, for the $k$th iteration ($k\in \mathbb{N}_+$), $\Psi_{k-1}(\mathbf{x}_{k-1}^{\ast}) \subseteq \chi$ is the feasible set of (P3) and $\mathbf{x}_k^{\ast} \in \Psi_{k-1}(\mathbf{x}_{k-1}^{\ast})$ is the optimal solution of (P3) with $f(\mathbf{x}_k^{\ast})\le f(\mathbf{x}_{k-1}^{\ast})$. This implies that each iterate is an improved solution that is feasible and continuing this process  
yields a non-increasing sequence: $\{f(\mathbf{x}_k^{\ast})\}_{k\in \mathbb{N}_+}$ that is bounded below by zero (since $f(\mathbf{x}) \ge 0 \,\, \forall \mathbf{x}\in\chi$). Thus, by the greatest-lower-bound property of real numbers, we know $\inf_{k\in \mathbb{N}_+}\{f(\mathbf{x}_k^{\ast})\} \in [0,f(\mathbf{x}_0)]$ exists. Since $\{f(\mathbf{x}_k^{\ast})\}_{k\in\mathbb{N}_+}$ is non-increasing and bounded below, 
by the monotone convergence theorem~\cite{rosenlicht1986introduction}, $error(k):=||f(\mathbf{x}_k^{\ast})-f(\mathbf{x}_{k-1}^{\ast})||_{\infty} \rightarrow 0$ as $k \rightarrow \infty$. Thus, we have proven that application of Algorithm~1 converges to an AC feasible, locally optimal solution, $\mathbf{x}^{\ast}$, through a sequence of successively improved AC-feasible iterates, $\mathbf{x}_k^{\ast}$, and that $\mathbf{x}^{\ast}$ improves on the original objective by $||f(\mathbf{x}_1^{\ast}) - f(\mathbf{x}^{\ast})||_{\infty}$.
\end{proof}

Since the above problem is convex in the continuous variables, the MILP can be solved effectively and provide good feasible solutions~\cite{lubin2016extended}. The formulation (P3) minimizes the utilization of reactive power from flexible DERs, prioritizing mechanical assets as a result. The formulation positions the voltage within the tighter voltage bounds, close to nominal, while the voltage violation terms $\mathbf{V^+_{\text{v}}}$ $\mathbf{V^-_{\text{v}}}$ ensure feasibility of the solution.

  In Section~\ref{sec:sim_results}, simulations involving standard IEEE test networks, e.g., see~\cite{kersting2001radial},  show that the results of this analysis holds for these radial distribution network. Simulation-based analysis is conducted on IEEE-13 node and IEEE-37 node system to check the validity of the approach on standard networks and analyze how tighter voltage bounds affect the utilization of DER reactive power at optimality.

\section{Simulation Results}\label{sec:sim_results}
 In this section, simulation tests are conducted on IEEE test cases and validation of the results is performed with Matpower~\cite{zimmerman2011matpower} on a standard MacBook Pro laptop with 2.2 GHz of processor speed and 16 GB RAM. Simulation results illustrate the validity of the VPO problem (P3). The optimization problem is solved using GUROBI 8.0~\cite{gurobi}, whereas the simulation is performed with AC load flows in Matpower. In all the simulation results, the system was solved to a MIP gap of under 0.01\%.
 
 \begin{figure}[h]
\centering
\includegraphics[width=0.33\textwidth]{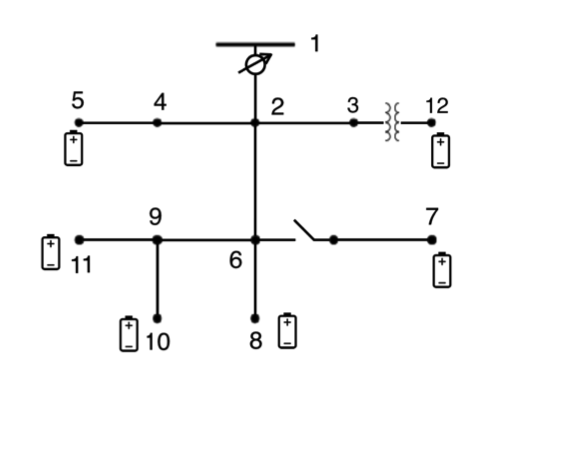}
\caption{\label{fig:IEEE_13node} IEEE-13 node distribution network with added DERs at leaf nodes.}
\end{figure}
 
For the IEEE-13 node shown in Fig.~\ref{fig:IEEE_13node} and IEEE-37 node test case shown in Fig.~\ref{fig:IEEE_37node}, optimal reactive dispatch schedules from (P3) are fed to an AC load flow in Matpower. The IEEE-13 node test case is modified to include capacitor banks at nodes 7 and 11, besides having an OLTC connecting nodes 3 and 12. Each capacitor bank operates with 10 increments with each increment being 50 kVAr. Apart from these mechanical resources, DERs are placed at leaf nodes $5$, $7$, $8$, $10$, $11$ and $12$ with each DER $q_{\text{g},i}$ at node $i$ having a range of -100 to +100 kVAr. In all the test cases the value of $\alpha$ is chosen to be .001.

The comparison of the output voltages of the optimizer (upper and lower bounds) and the AC power flow solver over two iterations are shown in Fig.~\ref{fig:voltprofile1} and~\ref{fig:voltprofile2} respectively. The optimal value changes from $3.6414\times 10^{-6}$ (pu) to $3.196\times 10^{-6}$ (pu), while the OLTC tap position stays fixed at position 2. From the figures, it is shown that the actual voltages are within the determined bounds and the voltage bounds converge to the AC power flow solution.


\begin{figure}
    \centering
  \subfloat[\label{fig:voltprofile1}]{%
       \includegraphics[width=0.5\linewidth]{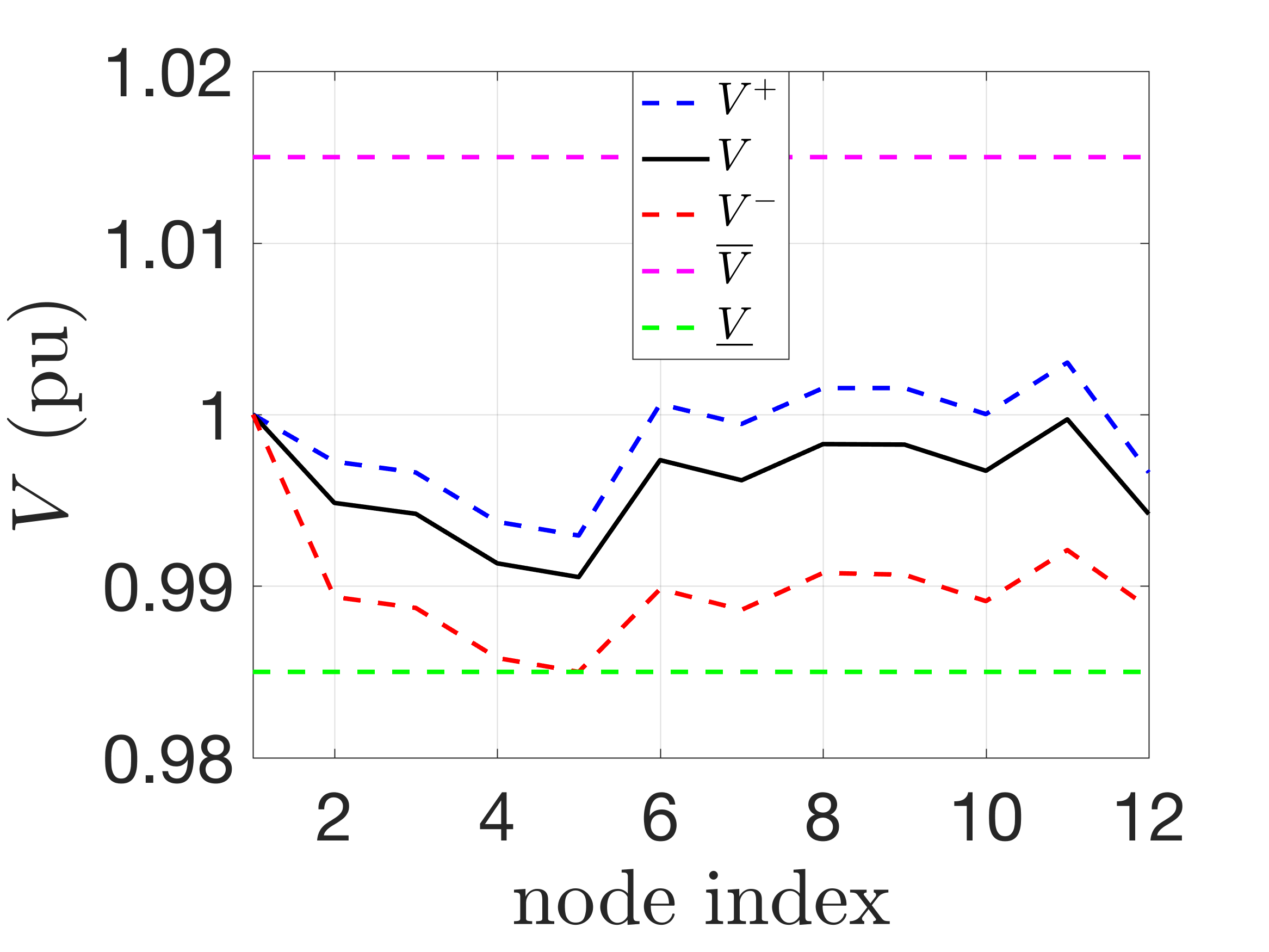}}
    \hfill
  \subfloat[\label{fig:voltprofile2}]{%
        \includegraphics[width=0.5\linewidth]{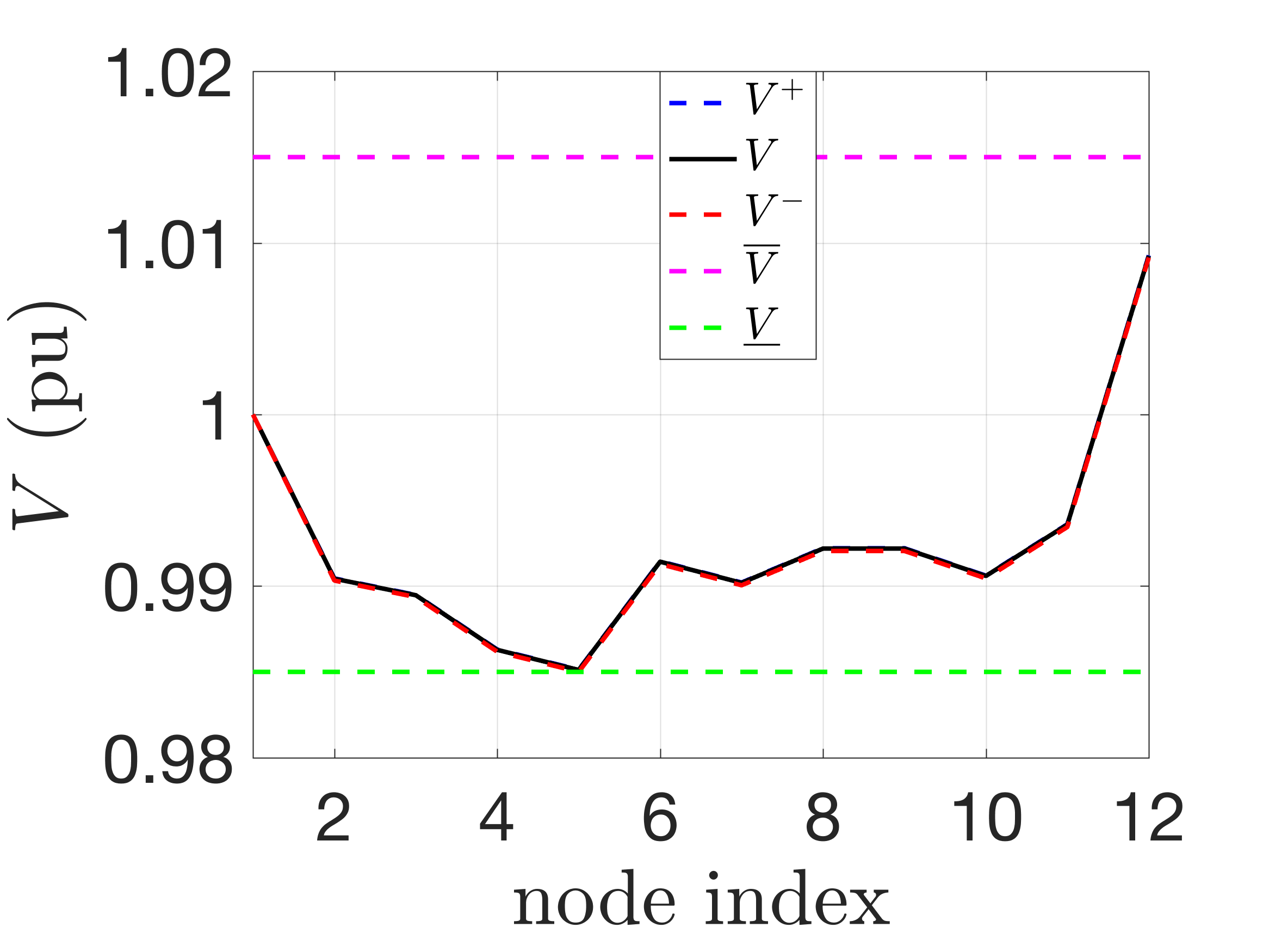}}

\caption{(a) Comparison of the actual nodal voltage with the upper and lower bound voltages for the first iteration. (b) Comparison of the actual nodal voltage with the upper and lower bound voltages for the second iteration. The objective value changed from 3.614$\times 10^-6$ to 3.196 $\times 10^-6$ with a value of $\alpha$ chosen being $0.001$ and the OLTC tap position at tap position 2 in both iterations.}
  \label{fig1} 
\end{figure}

In another simulation conducted on the IEEE-13 node system, we study the sensitivity on the reactive power utilization and the voltage violation term. Fig.~\ref{fig:alphaQlog} shows the variation in the total flexible reactive power consumption as the parameter $\alpha$ is varied, with the nominal value of $\alpha$ being highlighted.
Similarly, Fig.~\ref{fig:alphaVsllog} shows the change in the total voltage violation term as we sweep across $\alpha$, illustrating that as $\alpha$ increases the voltage violation term reduces as expected.

\begin{figure}
    \centering
  \subfloat[\label{fig:alphaQlog}]{%
       \includegraphics[width=0.5\linewidth]{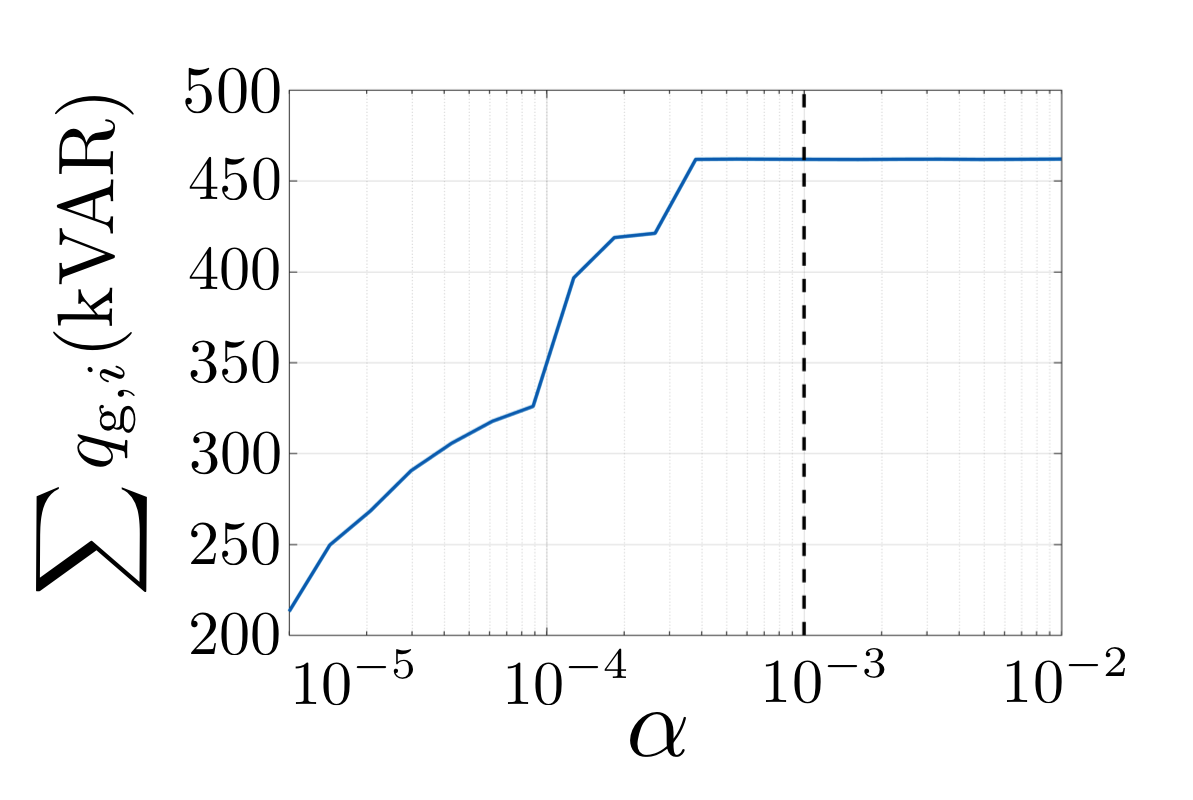}}
    \hfill
  \subfloat[\label{fig:alphaVsllog}]{%
        \includegraphics[width=0.5\linewidth]{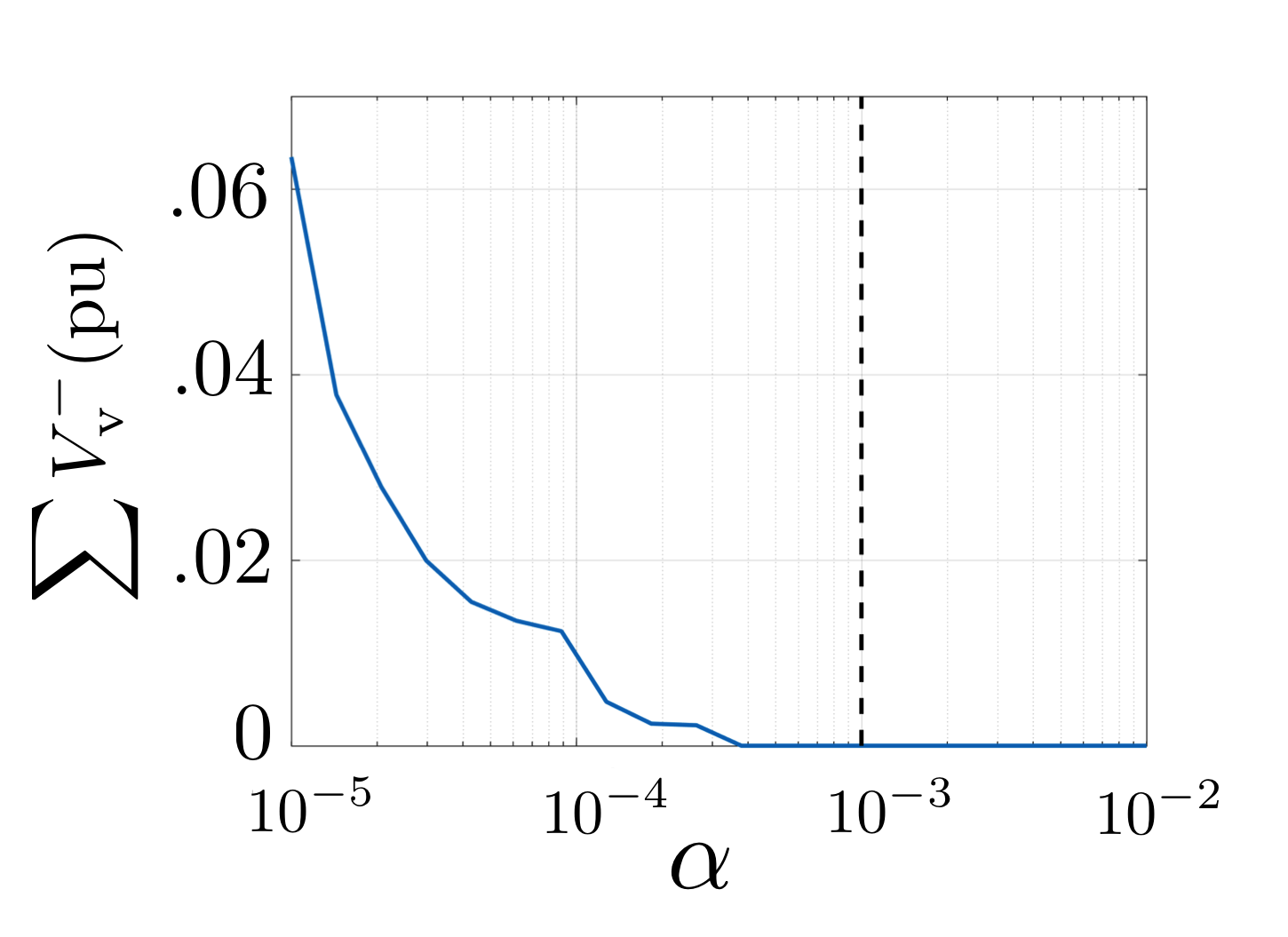}}

\caption{(a) Variation in the total flexible reactive power consumption with change in $\alpha$, (b) Variation in the total lower voltage violation term with change in $\alpha$.}
  \label{fig2} 
\end{figure}

Fig.~\ref{fig:VlowQ} shows the variation of the total flexible reactive power as the lower voltage bound is increased, showing the trade off between positioning voltage closer to nominal and the utilization of flexible reactive power resources, with the nominal value of $\mathbf{\underline{V}}$ used being highlighted. Similarly, Fig.~\ref{fig:VlowVsl} shows the variation in the total voltage violation with the increase in the lower voltage bound, showing a similar trend.

\begin{figure}
    \centering
  \subfloat[\label{fig:VlowQ}]{%
       \includegraphics[width=0.5\linewidth]{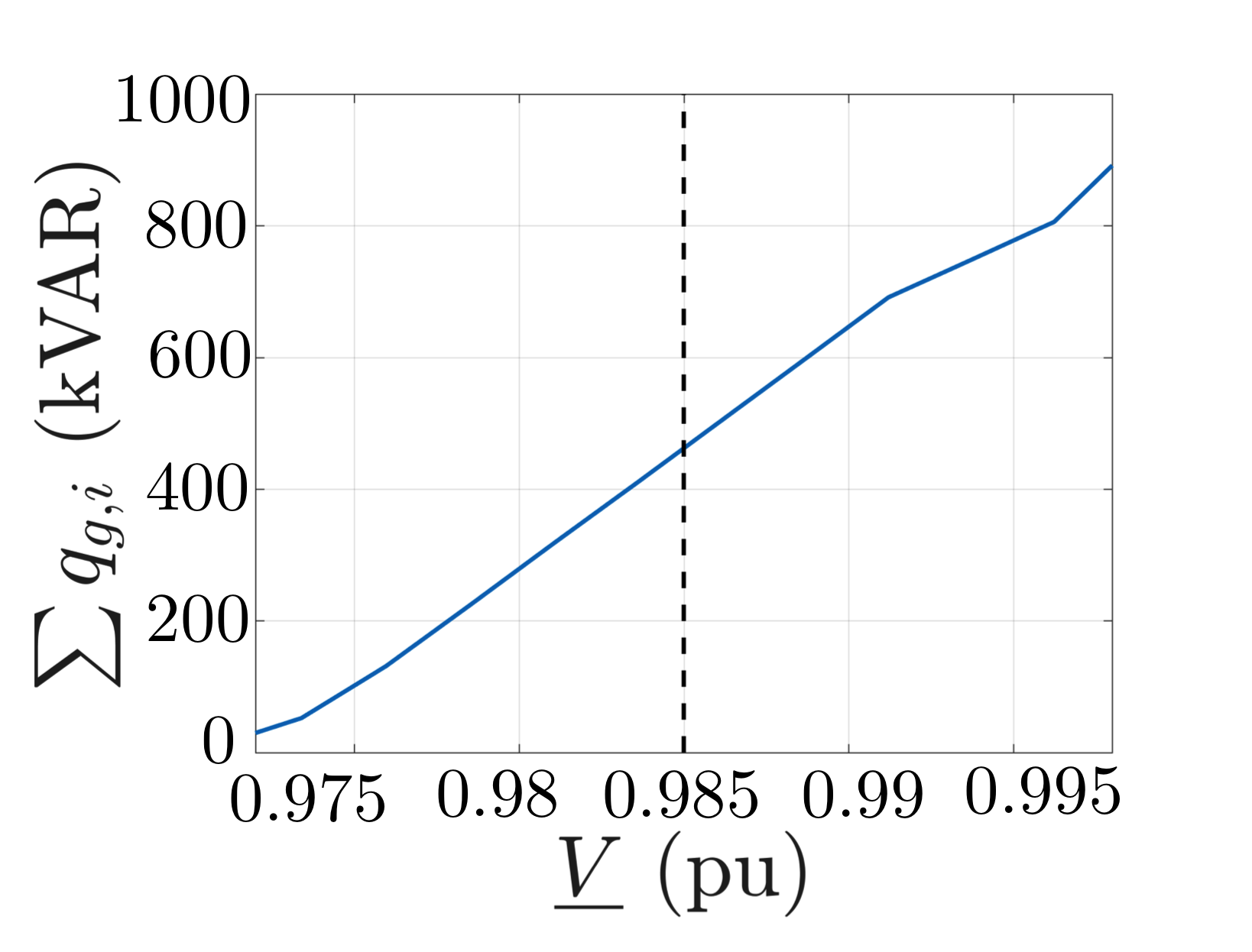}}
    \hfill
  \subfloat[\label{fig:VlowVsl}]{%
        \includegraphics[width=0.5\linewidth]{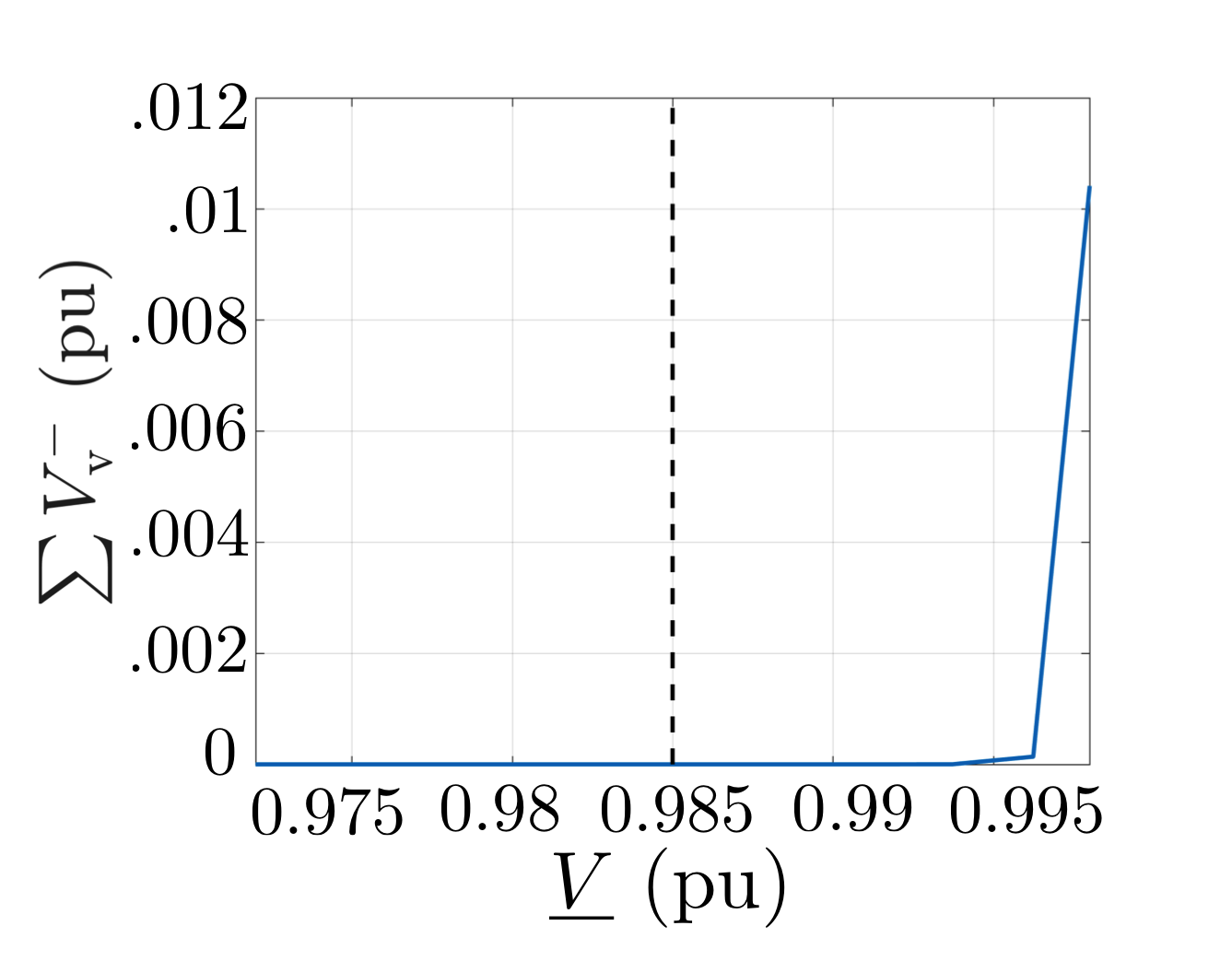}}

\caption{(a) Variation in the total flexible reactive power consumption with change in $\mathbf{\underline{V}}$, (b) Variation in the total lower voltage violation term with change in $\mathbf{\underline{V}}$.}
  \label{fig3} 
\end{figure}




Simulations are also conducted by considering a daily predicted 24-hour load profile shown in Fig.~\ref{fig:load_profile} obtained from real load data measured from feeders near Sacramento, CA during the month of August, 2012~\cite{bank2013development}.  Figure~\ref{fig:reactive_profle} shows the predicted aggregated reactive power supply from DERs over the horizon, whereas Fig.~\ref{fig:cap_profile} shows the predicted aggregate cap bank reactive power supply, illustrates that as the load in the system increases, the aggregate utilization of reactive power from capacitor banks also increases and follows a similar trend.  Figure~\ref{fig:xmer_profile} shows the optimal OLTC tap positions over the prediction horizon.

\begin{figure}
    \centering
  \subfloat[\label{fig:load_profile}]{%
       \includegraphics[width=0.5\linewidth]{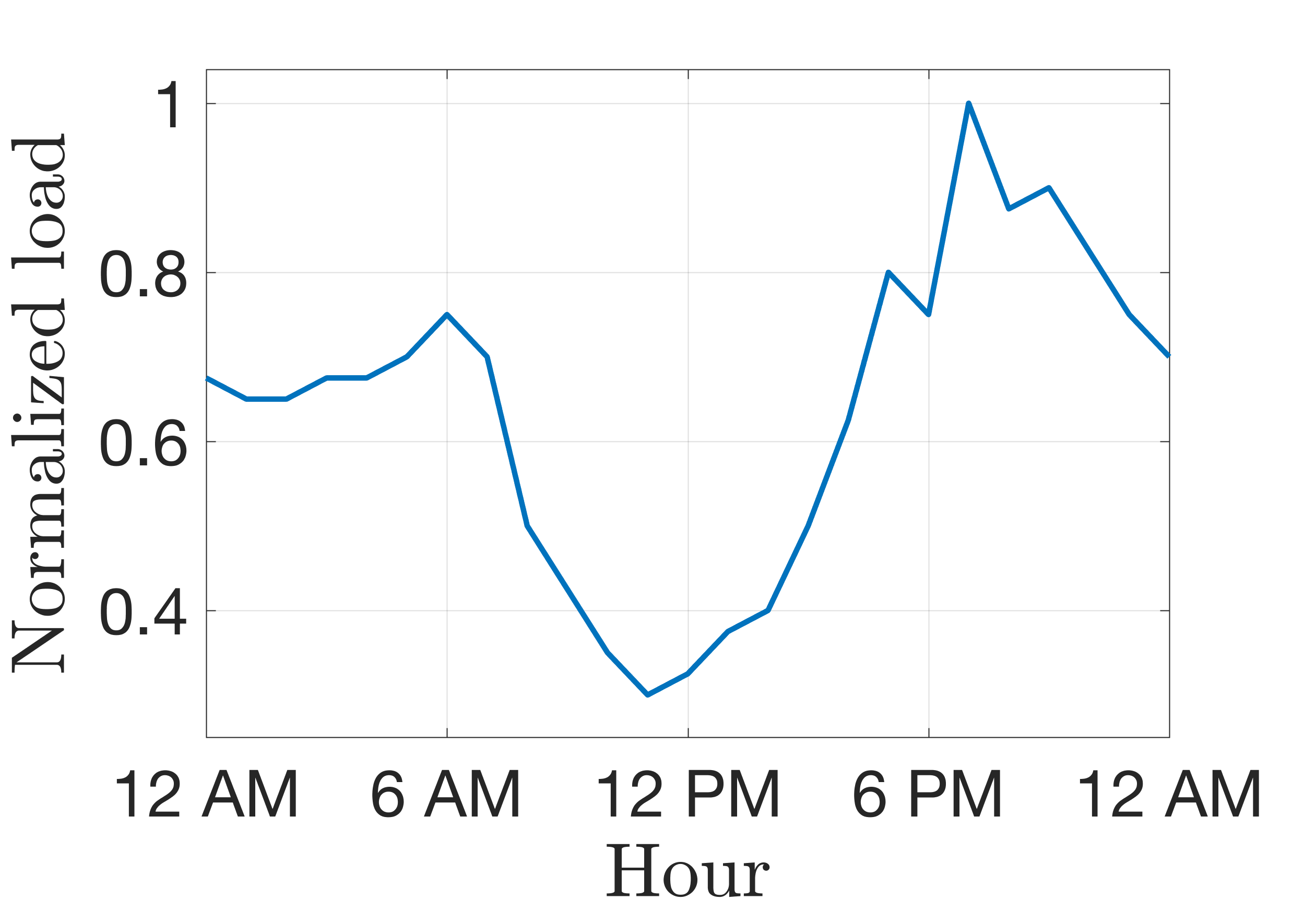}}
    \hfill
  \subfloat[\label{fig:reactive_profle}]{%
        \includegraphics[width=0.5\linewidth]{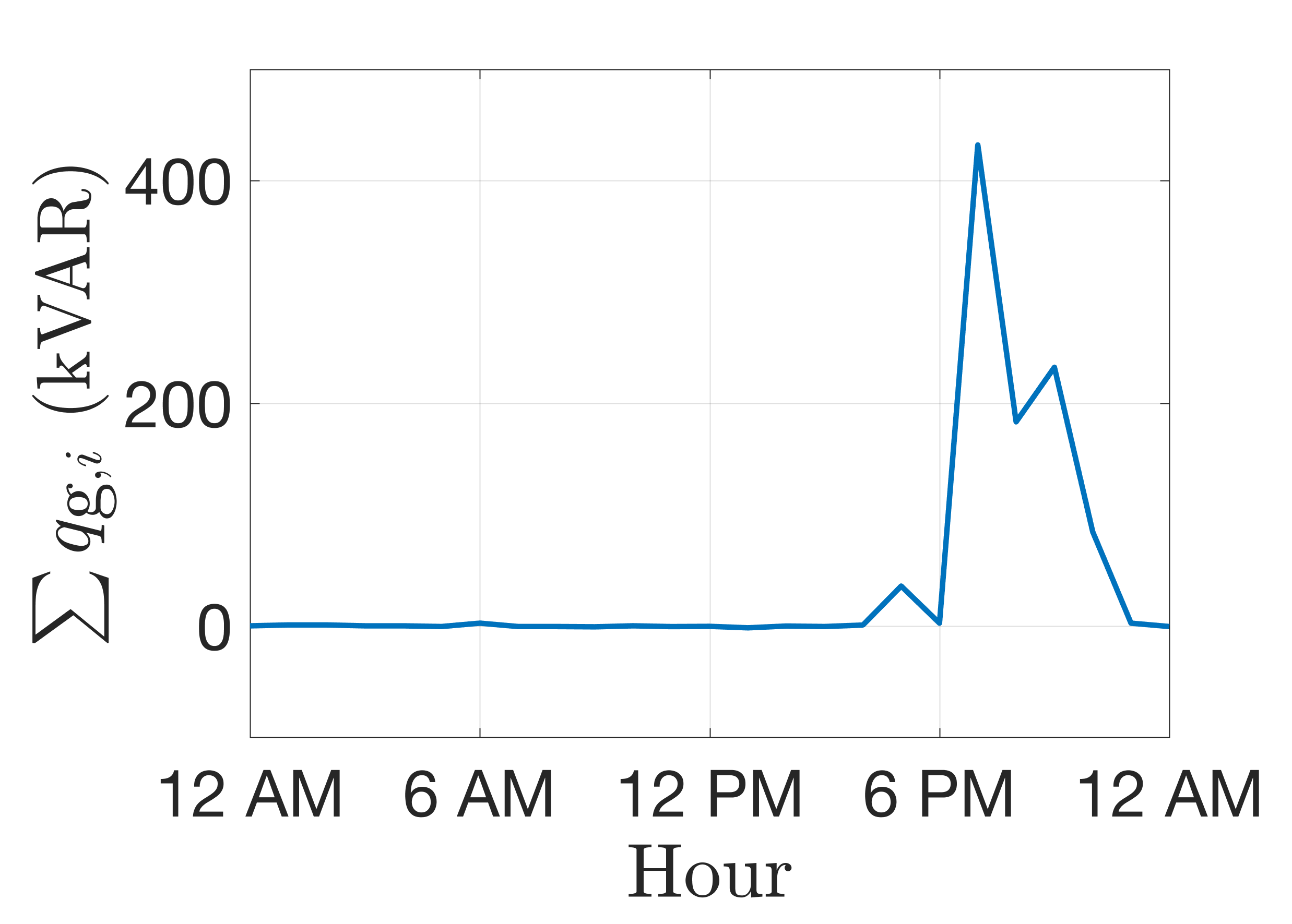}}
        \\
  \subfloat[\label{fig:cap_profile}]{%
        \includegraphics[width=0.5\linewidth]{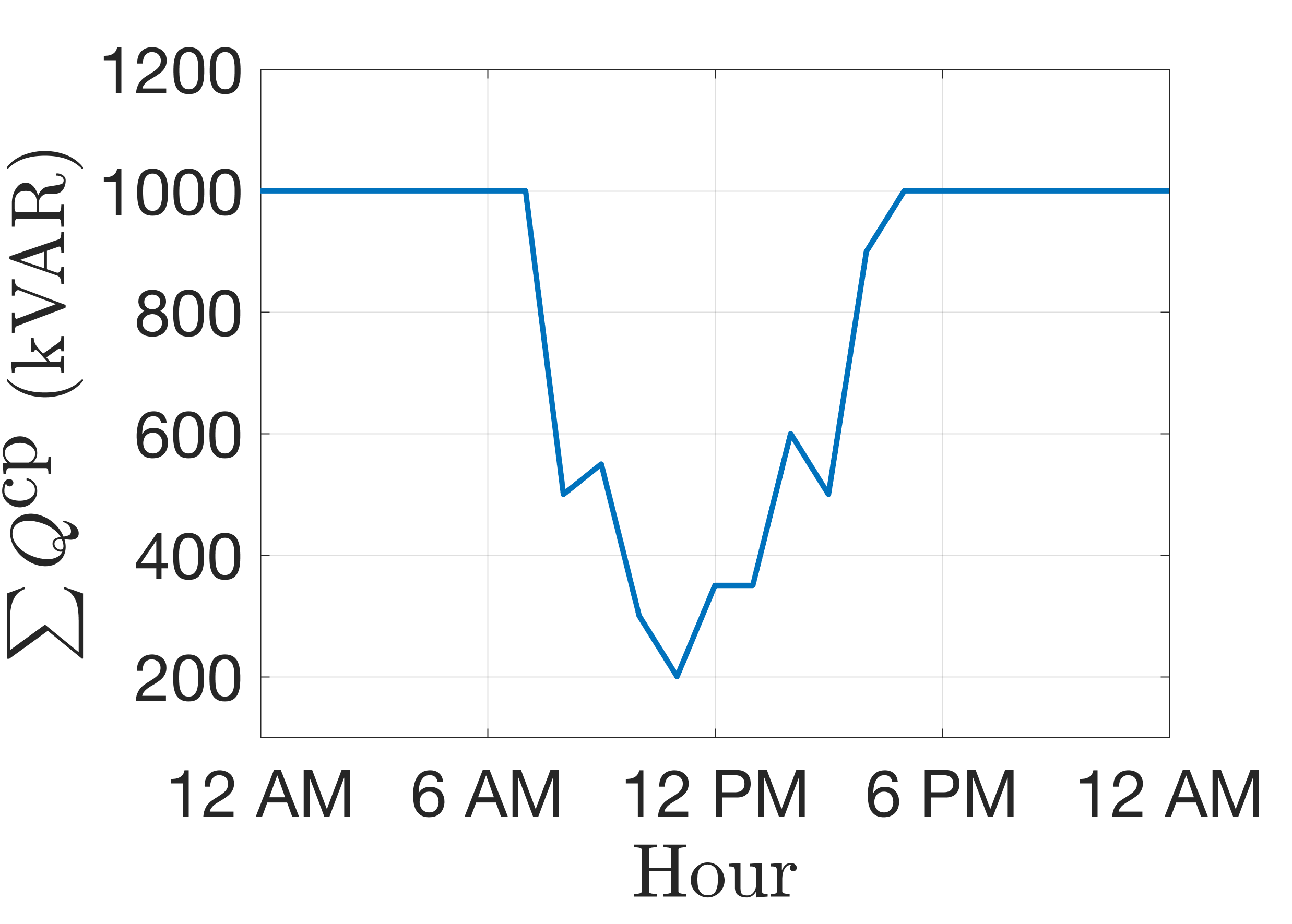}}
    \hfill
  \subfloat[\label{fig:xmer_profile}]{%
        \includegraphics[width=0.5\linewidth]{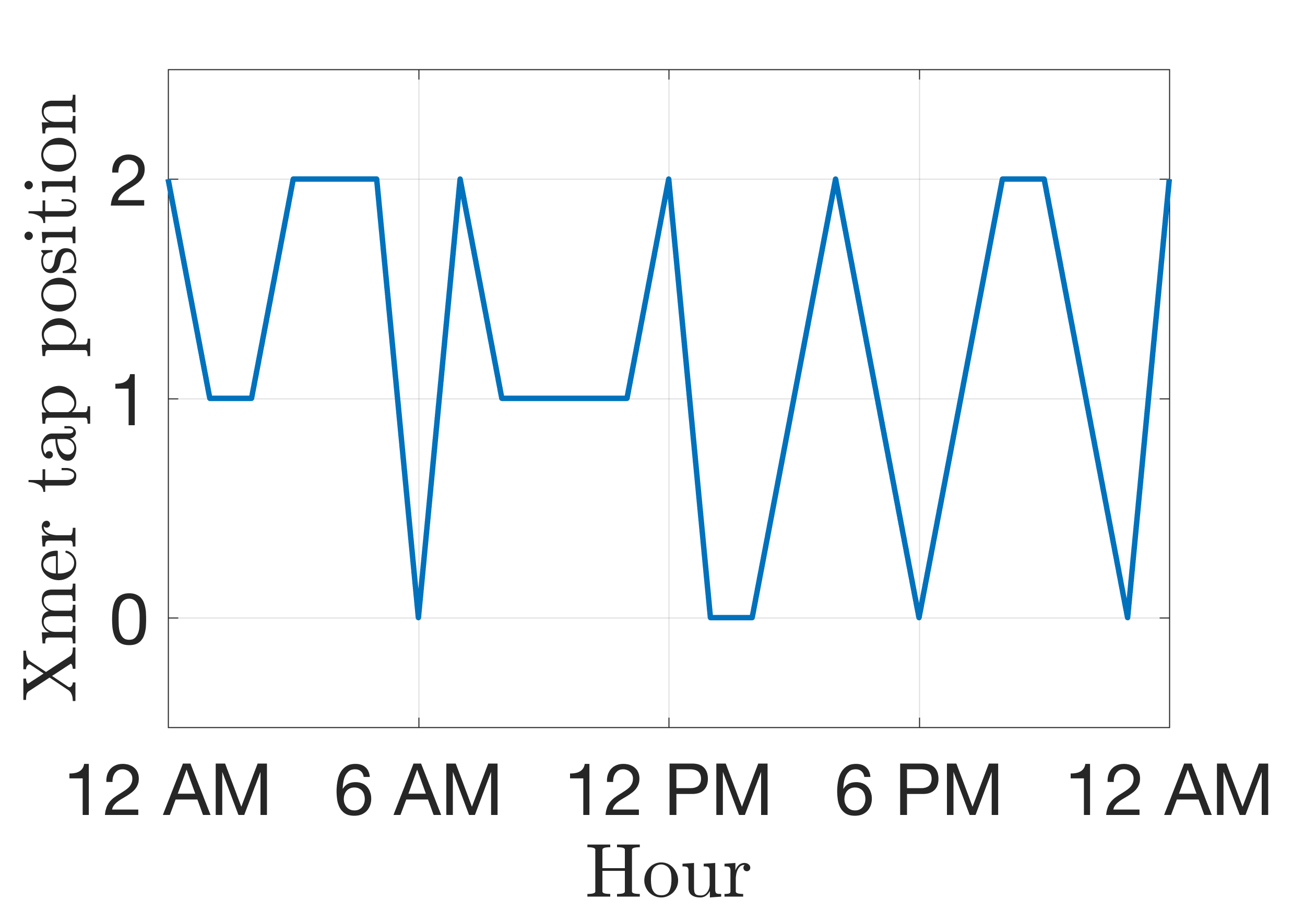}}
\caption{(a) Predicted 24-hour normalized load profile, (b) optimized schedule of aggregate reactive power from DERs utilized over the 24-hour horizon for IEEE-13 node system, (c) optimized schedule of the total reactive power from capacitor banks (at 1 p.u voltage) over the 24-hour horizon, (d) optimized OLTC tap position for IEEE-13 node system over the 24-hour horizon.}
  \label{fig4} 
\end{figure}



Fig.~\ref{fig:reac_comp_13node} shows the comparison between the reactive power supply between the case using capacitor banks and the case without the use of capacitor banks, showing that capacitor banks supply part of the reactive power reducing the burden on DERs. 
Fig.~\ref{fig:Vsl_comp_13node} shows the comparison between the voltage violation terms between the case using capacitor banks and the case without capacitor banks, showing that the voltage violations are very small and hence the system violations are within the specified tighter voltage bounds.
\begin{figure}[h]
\centering
\includegraphics[width=0.3\textwidth]{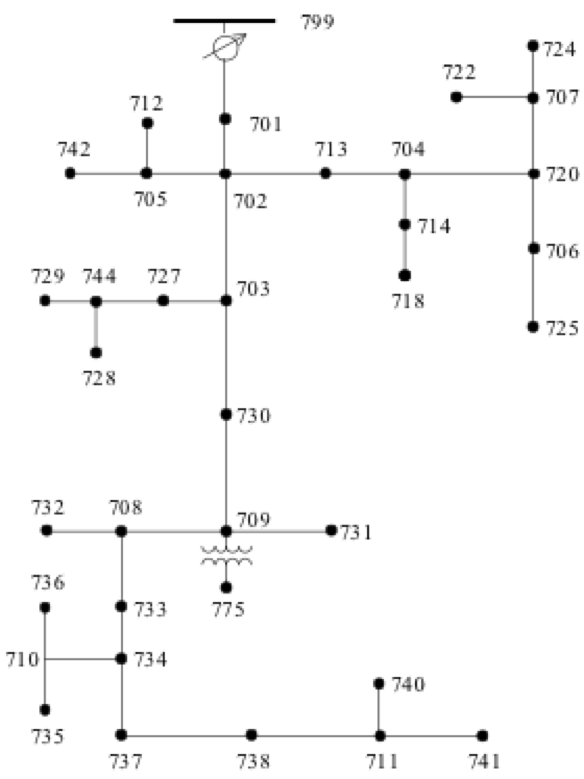}
\caption{\label{fig:IEEE_37node} IEEE-37 node distribution network~\cite{baker2018network}.}
\end{figure}

Further  simulations are conducted on IEEE-37 node  system shown in Fig.~\ref{fig:IEEE_37node}. In this case cap banks are positioned at nodes $724$, $725$, $728$, $732$, $736$ and $741$, whereas flexible DERs are placed at leaf nodes $714$, $731$, $734$, $744$ and $775$ and the load profile shown in Fig.~\ref{fig:load_profile} is used. For the IEEE-37 node system also, each capacitor bank operates with 10 increments with each increment being 10 kVAr and each DER has a range of -100 to +100 kVAr. Similar results are observed with regards to the reduction in reactive power utilization from flexible resources with the inclusion of capacitor banks as shown in Fig.~\ref{fig:reac_comp_37node}. Similarly, Fig.~\ref{fig:Vsl_comp_37node} shows the comparison of the voltage violation terms illustrating that the tighter voltage bounds are maintained. Further, Fig.~\ref{fig:xmer_profile_37} shows the optimal OLTC tap positions over the prediction horizon, whereas Fig.~\ref{fig:volt_profile_37} shows the comparison between voltages obtained from the optimzer (upper and lower bound) and Matpower at node 775 of the IEEE-37 node system over the prediction horizon.

Finally, Fig.~\ref{fig:solvetime_37} shows the increase in solve time (over one iteration) as the number of devices increases in the network. It can be seen that the convex formulation scales well with the increase in problem size.


\begin{figure}
    \centering
  \subfloat[\label{fig:reac_comp_13node}]{%
       \includegraphics[width=0.5\linewidth]{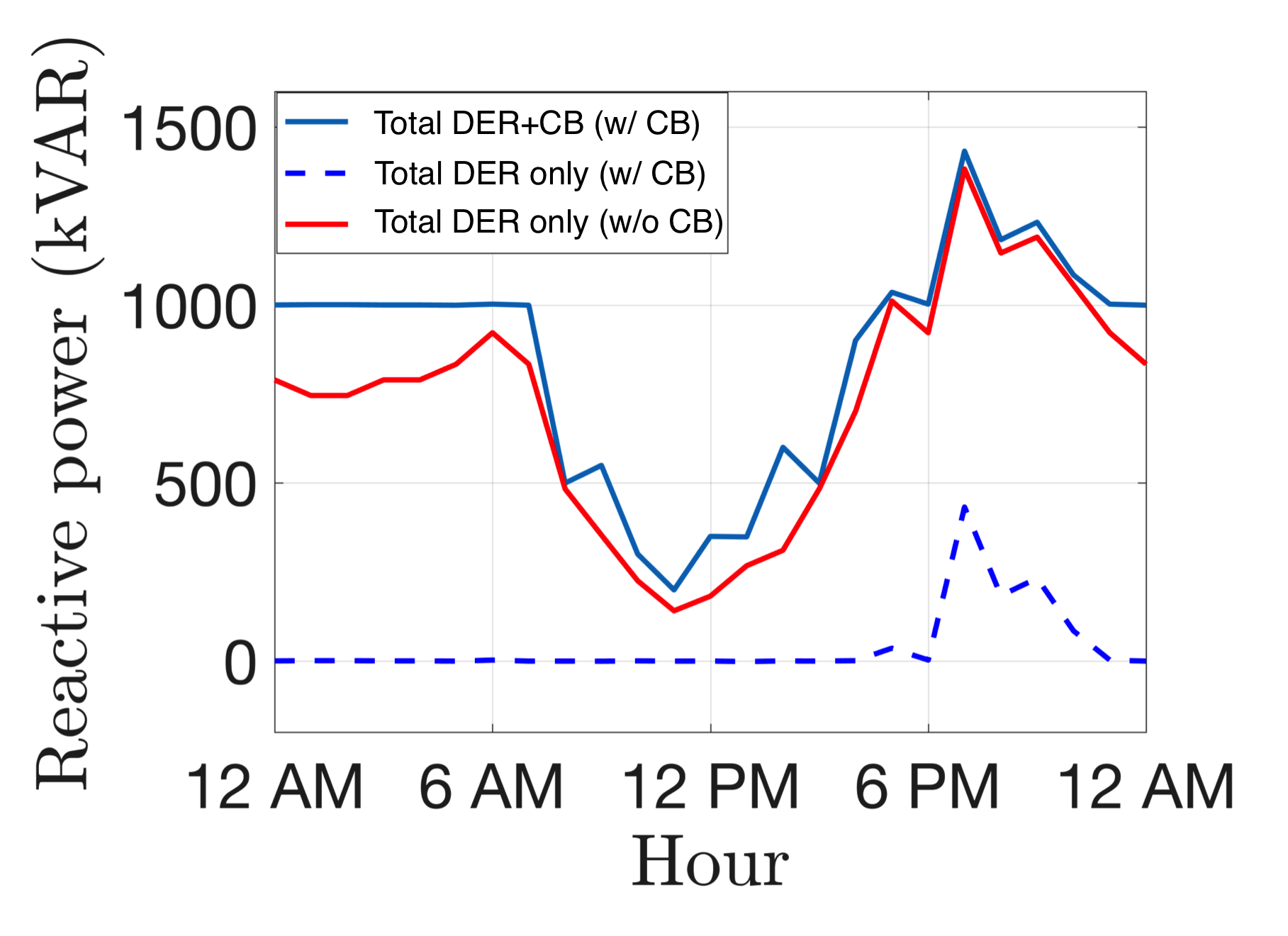}}
    \hfill
  \subfloat[\label{fig:Vsl_comp_13node}]{%
        \includegraphics[width=0.5\linewidth]{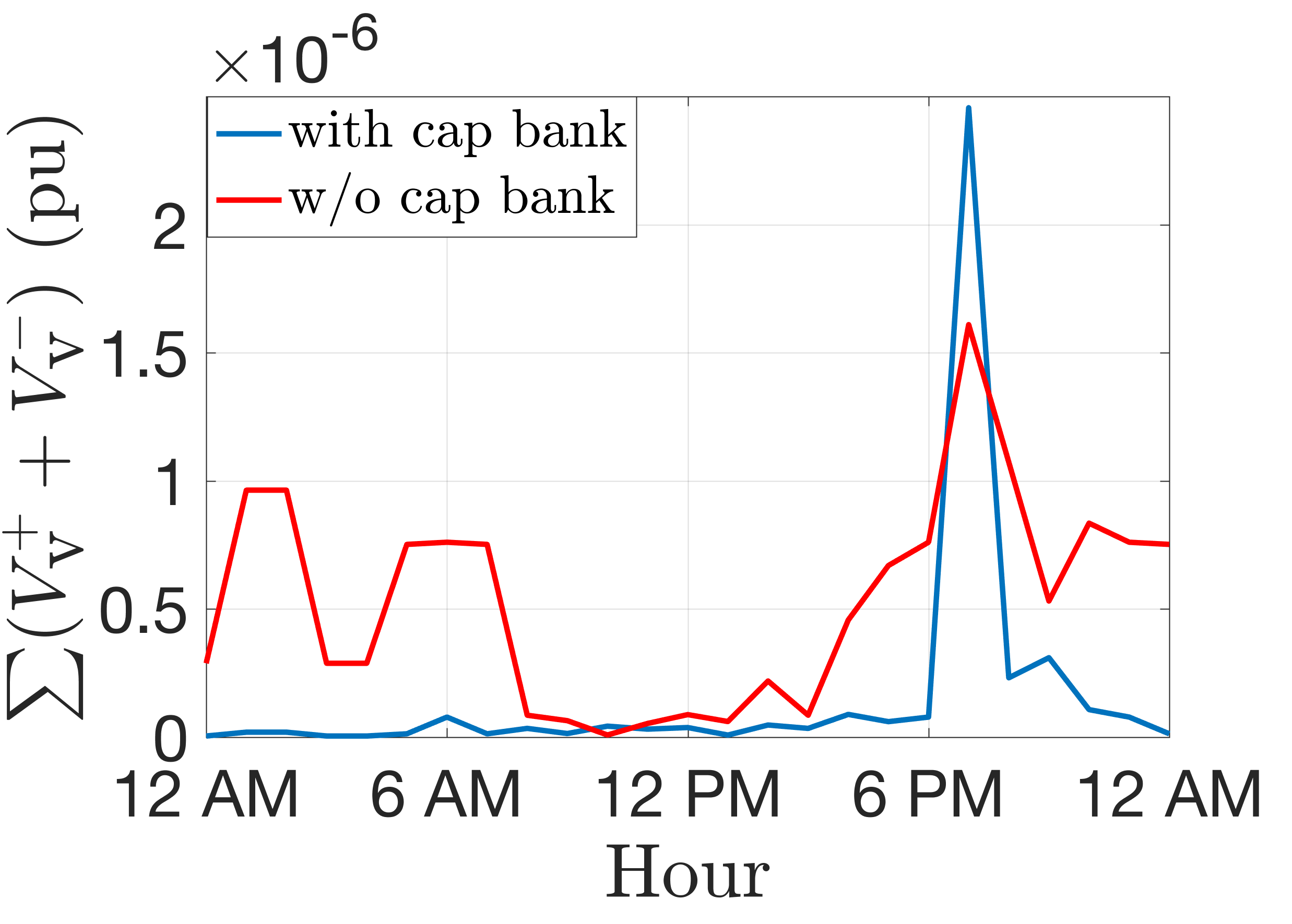}}

\caption{(a) Comparison of reactive power supply between the case using capacitor banks and without capacitor banks for IEEE-13 node system. The figure compares the total reactive power supply from DERs and cap banks (Total DER+CB (w/ CB)) with the reactive power supply only from DERs when cap banks are utilized (Total DER only (w/ CB)) and with DER reactive power supply when cap banks are not utilized (Total DER only (w/o CB)) , (b) Comparison of total voltage violation over time between the case using capacitor banks and without capacitor banks for IEEE-13 node system.}
  \label{fig5} 
\end{figure}

\begin{figure}
    \centering
  \subfloat[\label{fig:reac_comp_37node}]{%
       \includegraphics[width=0.5\linewidth]{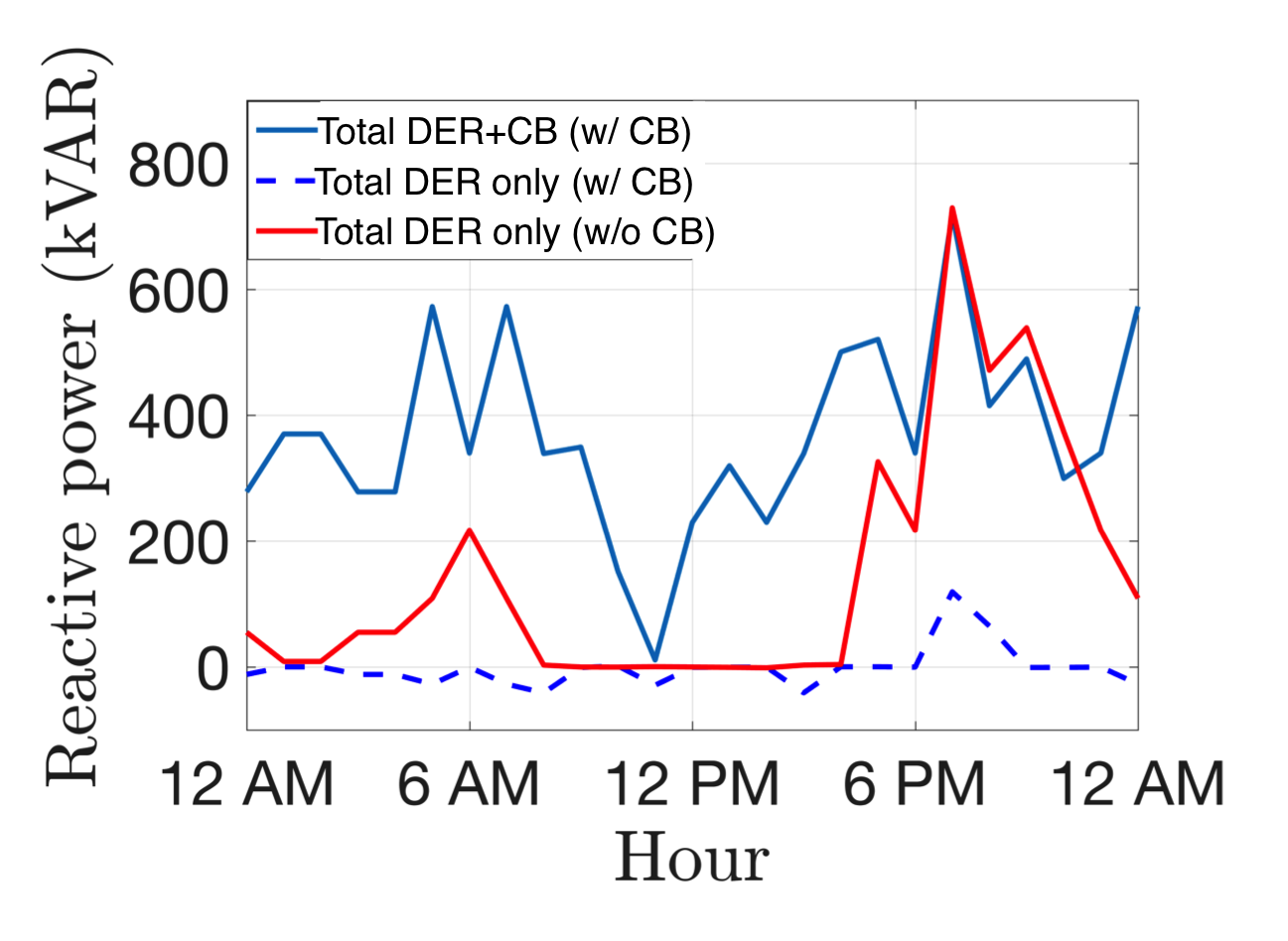}}
    \hfill
  \subfloat[\label{fig:Vsl_comp_37node}]{%
        \includegraphics[width=0.5\linewidth]{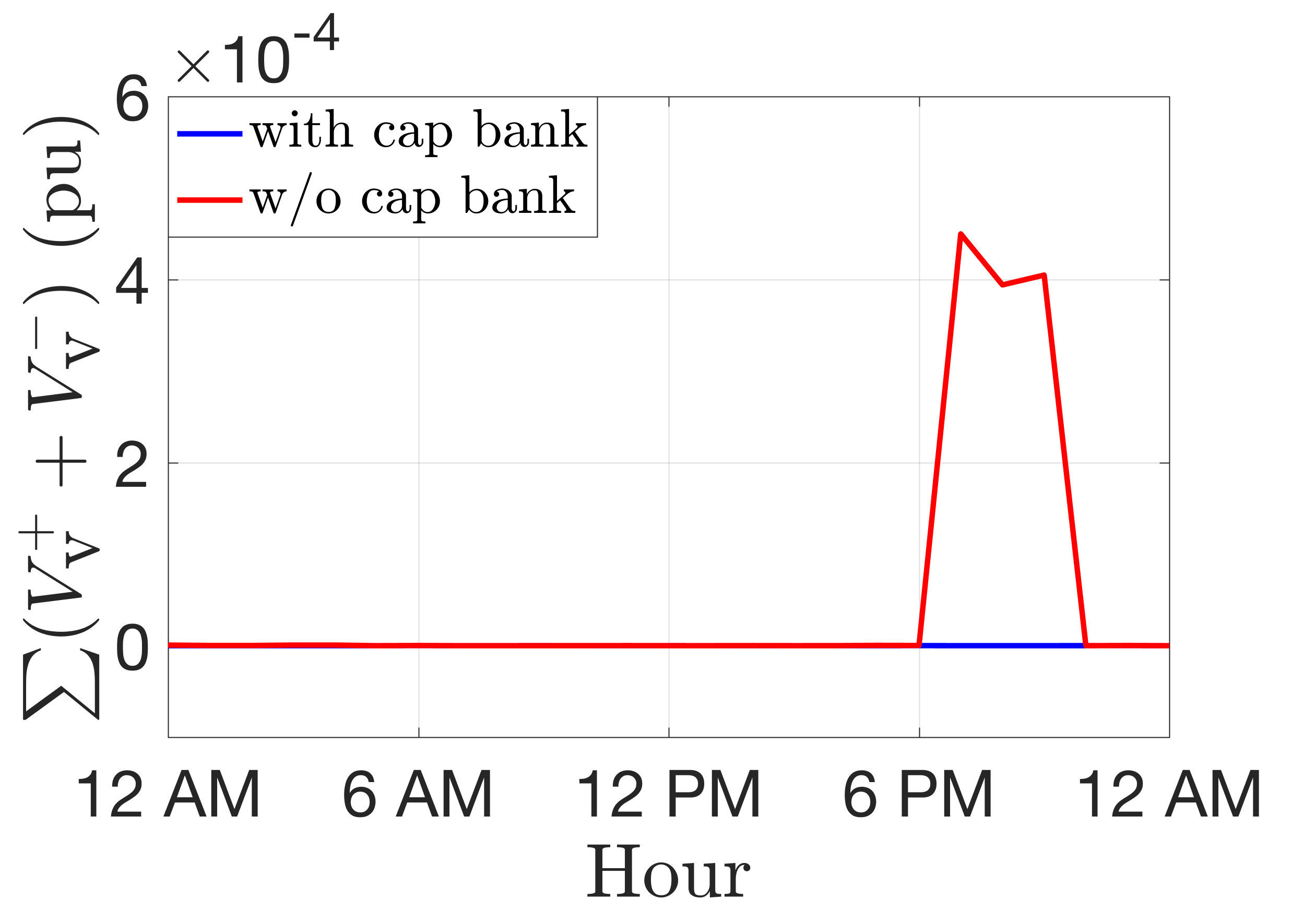}}

\caption{(a) Comparison of reactive power supply between the case using capacitor banks and without capacitor banks for IEEE-37 node system. The figure compares the total reactive power supply from DERs and cap banks (Total DER+CB (w/ CB)) with the reactive power supply only from DERs when cap banks are utilized (Total DER only (w/ CB)) and with DER reactive power supply when cap banks are not utilized (Total DER only (w/o CB)) , (b) Comparison of total voltage violation over time between the case using capacitor banks and without capacitor banks for IEEE-37 node system.}
  \label{fig6} 
\end{figure}



\begin{figure}
    \centering
    \subfloat[\label{fig:xmer_profile_37}]{%
        \includegraphics[width=0.5\linewidth]{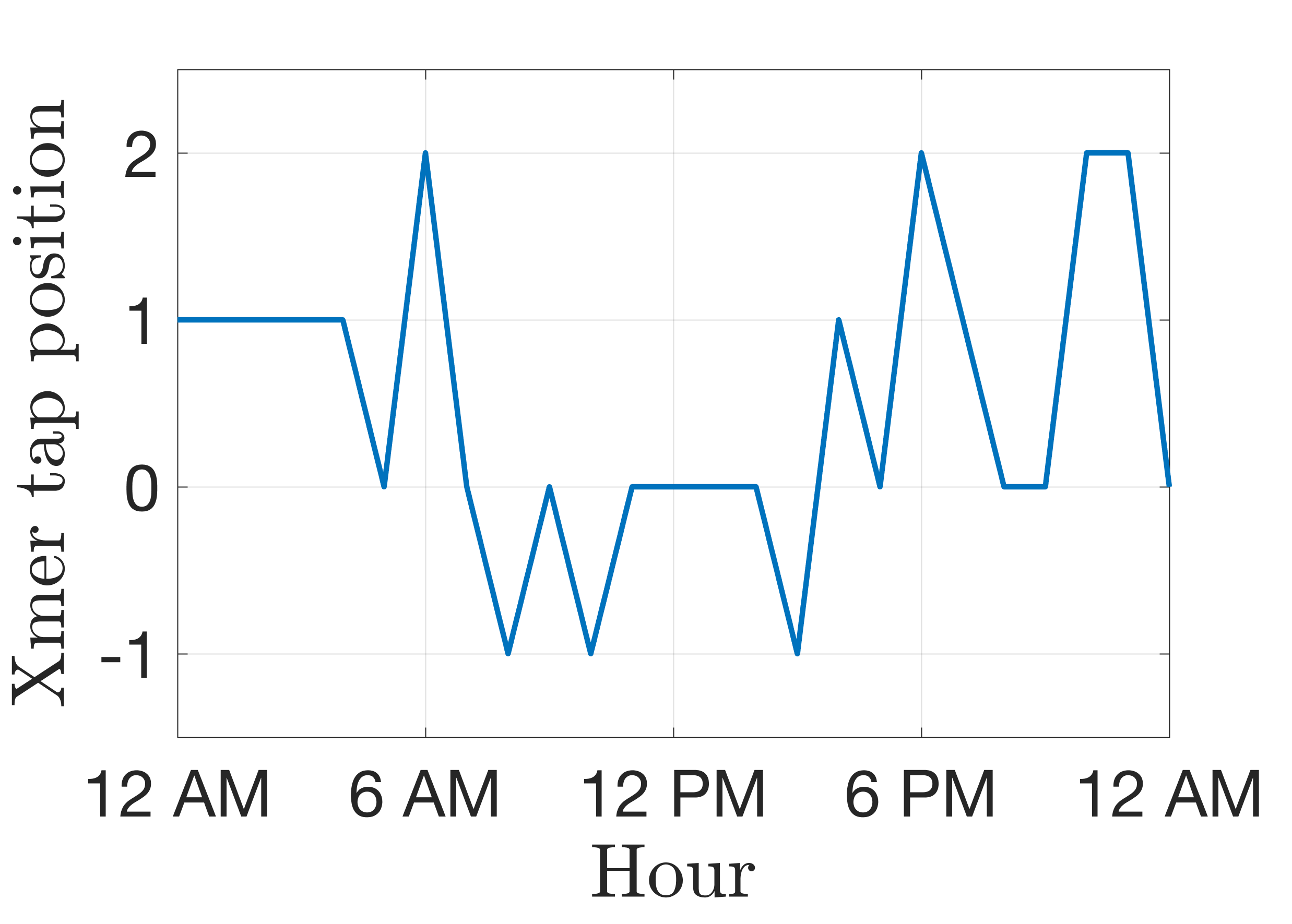}}
    \hfill
  \subfloat[\label{fig:volt_profile_37}]{%
       \includegraphics[width=0.5\linewidth]{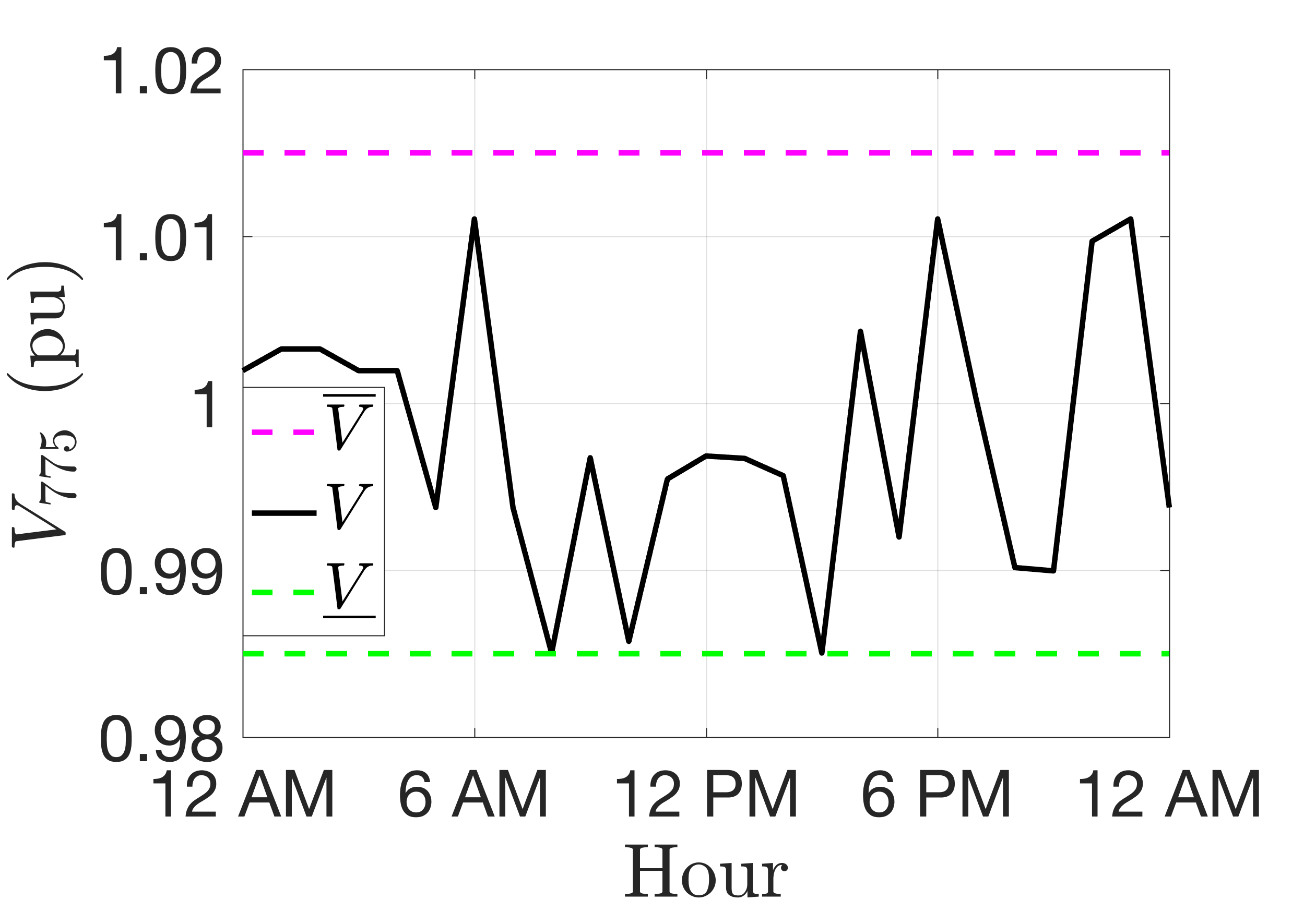}}

\caption{(a) Predicted OLTC tap position for IEEE-37 node system over the 24-hour horizon. (b) Comparison of the actual nodal voltage at node 775 with the upper and lower bounds over a 24-hour horizon for IEEE-37 node system }
  \label{fig7} 
\end{figure}
\begin{figure}[h]
\centering
\includegraphics[width=0.35\textwidth]{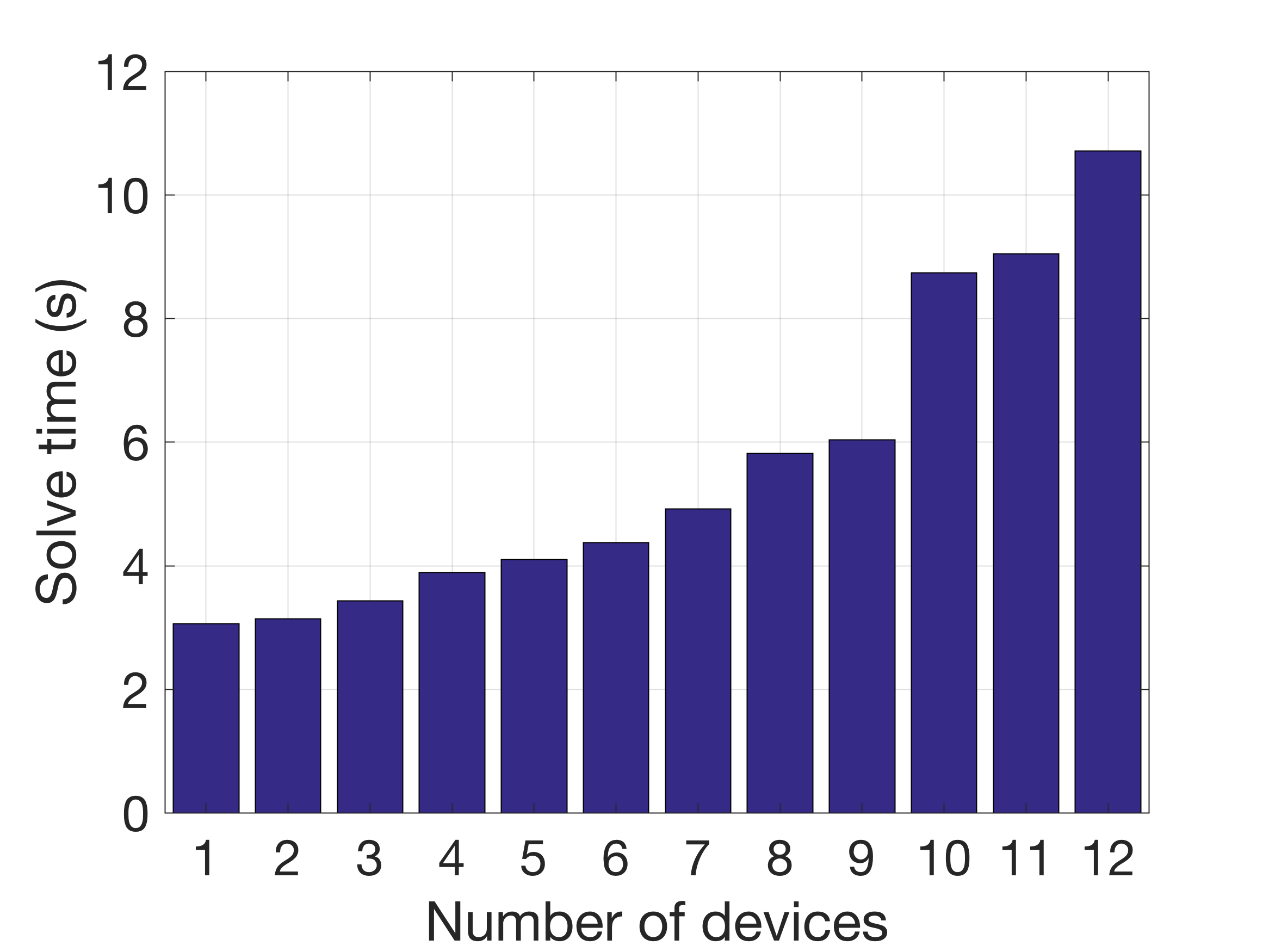}
\caption{\label{fig:solvetime_37} Solve time for IEEE-37 node system with number of devices (capacitor banks) in the network over one iteration.}
\end{figure}

Through these simulation results, it is observed that the voltage positioning algorithm results in a network admissible solution that prioritizes the utilization of mechanical assets over flexible resources to position the voltage close to nominal and hence could be utilized for voltage control of mechanical assets in distribution networks.

\section{Conclusions and future work}\label{sec:conclusion}
This paper introduces a holistic voltage positioning algorithm to optimally schedule mechanical switching devices, such as on-load tap changing transformers and capacitor banks, together with more responsive DERs in a distribution grid. The optimization program makes maximum use of mechanical resources to position the voltage close to nominal using tighter inner voltage bounds to counter the predicted hourly variation of renewable generation. At the same time, the scheduling of responsive reactive resources from DERs is reduced, making them available at the faster time-scale to counter fast minute-to-minute variation inherent to renewable generation. The optimization problem is formulated as a MILP through an convex inner approximation of the OPF ensuring network admissible solutions. 
 The optimization problem is validated via simulations on the IEEE-13 node and IEEE-37 node test feeders 
and the results are compared with AC load flows from Matpower. 
 The results validate the approach.

Future work will focus on the multi-period extension by considering the time-coupling introduced due to energy storage and ramp-rate limits. Of particular interest is the problem of restricting the frequent tap changes of OLTCs in distribution networks that has been previously highlighted in~\cite{kersulis2016renewable}. The VPO formulation is well-suited for limiting the switching of tap-changers by introducing tap change rate constraints in the multi-period optimization formulation.
Furthermore, we are seeking to extend the results to consider the effects of uncertainty in the VPO.

It is important to note that this current manuscript considers radial, balanced, and inductive distribution feeders. However, realistic distribution feeders are sometimes meshed, often unbalanced, and usually a mix of inductive and capacitative lines, which means that extending this work to a full, three-phase AC formulation is valuable towards utility practice. In the power systems literature, the challenges associated with the optimization of more general unbalanced systems has been widely discussed and represent challenging, open technical problems~\cite{nazir2018receding, stevenLow_unbal,nazir_inner, nando_unbal,Dan_convex_rest}. Thus, we are interested in leveraging these recent results to extend the convex inner approximation formulation to unbalanced and meshed networks in future work.

\bibliographystyle{IEEEtran}
\small\bibliography{fix.bib}
\end{document}